\documentclass[11pt]{amsart}
\usepackage{amsmath}
\usepackage{amsfonts}
\usepackage{amssymb}
\usepackage{mathrsfs}
\usepackage{xcolor}

\newtheorem{teo}{Theorem}[section]
\newtheorem{lem}[teo]{Lemma}
\newtheorem{cor}[teo]{Corollary}
\newtheorem{prop}[teo]{Proposition}

\newtheorem{es}[teo]{Example}

\newtheorem{defi}[teo]{Definition}
\newtheorem{oss}[teo]{Remark}

\newcommand{\ud}{\mathrm{d}}

\newcommand{\mm}{\mathfrak m}

\newcommand{\supp}{\mathrm{supp}}

\renewcommand{\Phi}{\varPhi}
\newcommand{\R}{{\mathbb{R}}}

\newcommand{\bbR}{{\mathbb{R}}}
\newcommand{\bbN}{{\mathbb{N}}}

\newcommand{\average}{{\mathchoice {\kern1ex\vcenter{\hrule height.4pt
width 6pt
depth0pt} \kern-9.7pt} {\kern1ex\vcenter{\hrule height.4pt width 4.3pt
depth0pt}
\kern-7pt} {} {} }}
\newcommand{\ave}{\average\int}

\newcommand{\Leb}[1]{{\mathscr L}^{#1}}

\newcommand{\PI}{{\mathsf{PI}}}

\def\XXint#1#2#3{{\setbox0=\hbox{$#1{#2#3}{\int}$}
\vcenter{\hbox{$#2#3$}}\kern-.5\wd0}}

\numberwithin{equation}{section}

\title{Weighted Sobolev spaces on metric measure spaces}
\author[Luigi Ambrosio]{Luigi Ambrosio}
\address[Luigi Ambrosio]{Scuola Normale Superiore, Piazza dei Cavalieri 7, Pisa 56126, Italy}
\email[Luigi Ambrosio]{Luigi.Ambrosio@sns.it}
\author[Andrea Pinamonti]{Andrea Pinamonti}
\address[Andrea Pinamonti]{Dipartimento di Matematica, Universita di Bologna, Piazza di Porta San Donato 5, Bologna 40126, Italy}
\email[Andrea Pinamonti]{Andrea.Pinamonti@gmail.com}
\author[Gareth Speight]{Gareth Speight}
\address[Gareth Speight]{Department of Mathematical Sciences, University of Cincinnati, 2815 Commons Way, Cincinnati, OH 45221, United States}
\email[Gareth Speight]{Gareth.Speight@uc.edu}

\begin{document}

\begin{abstract}
We investigate weighted Sobolev spaces on metric measure spaces $(X,\ud,\mm)$. Denoting by $\rho$ the weight function,
we compare the space $W^{1,p}(X,\ud,\rho\mm)$ (which always concides with the closure $H^{1,p}(X,\ud,\rho\mm)$ of Lipschitz functions)
with the weighted Sobolev spaces $W^{1,p}_\rho(X,\ud,\mm)$ and $H^{1,p}_\rho(X,\ud,\mm)$ defined as in the Euclidean theory of weighted Sobolev spaces. 
Under mild assumptions on the metric measure structure and on the weight we show that  
$W^{1,p}(X,\ud,\rho\mm)=H^{1,p}_\rho(X,\ud,\mm)$. We also adapt the results in \cite{Muk} and in the recent paper \cite{Zhi}
to the metric measure setting, considering appropriate conditions on $\rho$ that ensure the equality 
$W^{1,p}_\rho(X,\ud,\mm)=H^{1,p}_\rho(X,\ud,\mm)$.
\end{abstract}

\maketitle
\section{Introduction}

The theory of Sobolev spaces $W^{1,p}(X,\ud,\mm)$ with $p\in (1,\infty)$ on metric measure spaces $(X,\ud,\mm)$ has by now reached a mature stage, after the seminal papers
\cite{Che}, \cite{Sha}, the more recent developments in \cite{AGS12} and the monographs \cite{BJ}, \cite{ShaK}. In this context, it is natural to investigate to what extent the Sobolev space is sensitive to the reference measure
$\mm$. It is clear that the measure $\mm$ is involved, since we impose $L^p(\mm)$ summability of the weak gradient, but things are more subtle. Indeed,
the measure $\mm$ is also involved in the definition of $(p,\mm)$-modulus $\mathrm{Mod}_{p,\mm}$ (Definition~\ref{def:modulus}) which, in turn, 
plays a role in the axiomatization in \cite{Sha}: by definition, $f\in W^{1,p}(X,\ud,\mm)$ 
if there exist a representative $\widetilde{f}\in L^p(\mm)$ of $f$ and $g\in L^p(\mm)$ such that
$$
|\widetilde{f}(\gamma(1))-\widetilde{f}(\gamma(0))|\leq\int_0^1 g(\gamma(t))|\dot\gamma(t)|\ud t
$$
along $\mathrm{Mod}_{p,\mm}$-a.e. absolutely continuous curve $\gamma:[0,1]\to X$. If such a function $g$ exists, then there is one with minimal $L^{p}(\mm)$ norm which is called the minimal gradient.

The definition adopted in \cite{Che}, instead, is equivalent but based on the approximation in $L^p(\mm)$ with functions having an upper gradient in $L^p(\mm)$.
More recently, in \cite{AmbGiSav,AGS12} it has been proved that $W^{1,p}(X,\ud,\mm)=H^{1,p}(X,\ud,\mm)$, where the latter space is defined
as the collection of all $L^p(\mm)$ functions for which there exist $f_n\in{\rm{Lip}}(X)\cap L^p(\mm)$ with $\int_X|f_n-f|^p\ud\mm\to 0$
and 
$$
\sup_{n\in\mathbb N}\int_X|\nabla f_n|^p\ud\mm<\infty
$$
(here and in the sequel $|\nabla g|$ denotes the local Lipschitz constant of $g$). In this case, following \cite{Che},
the minimal gradient is defined by considering all functions larger than weak $L^p(\mm)$ limits of local Lipschitz constants 
of sequences of Lipschitz functions converging to $f$ in $L^{p}(\mm)$, and considering the function with smallest $L^p(\mm)$ norm.
This general ``$H~=~W$'' result does not depend on structural assumptions on the metric measure structure: $(X,\ud)$ complete
and separable and $\mm$ a locally finite Borel measure are sufficient for the validity of this identification theorem. In this introduction
we shall denote by $|\nabla f|_w$ the minimal gradient arising from both unweighted $W$ and $H$ definitions, not emphasizing its potential dependence on $p$ (see \cite{DmSp}).

Given a Borel weight function $\rho:X\to [0,\infty]$,
in this paper we compare spaces $H~=~W$ relative to the metric measure structure $(X,\ud,\rho\mm)$ with the weighted spaces
built as in the Euclidean theory (namely $X=\R^n$, $\ud=$Euclidean distance, $\mm=\Leb{n}$, the
Lebesgue measure in $\R^n$). The first weighted space is
\begin{equation}\label{eq:Alica4}
W^{1,p}_\rho:=\left\{f\in W^{1,1}(X,\ud,\mm):\ |f|+|\nabla f|_w\in L^p(\rho\mm)\right\}
\end{equation}
endowed with the norm
$$
\|f\|_\rho:=\biggl(\int_X|f|^p\rho \ud\mm+\int_X|\nabla f|_w^p\rho\ud\mm\biggr)^{1/p},
$$
where $|\nabla f|_w$ is the minimal $1-$weak gradient of $f$ with respect to the unweighted space $(X,\ud,\mm)$.
Minimal regularity requirements (which provide respectively local finiteness of $\rho\mm$ and a basic embedding in $W^{1,1}$) are that 
$\rho\in L^1_{{\rm loc}}(\mm)$ and $\rho^{-1}\in L^{1/(p-1)}(\mm)$. 
%We shall
%always make the second assumption on $\rho$, adding the first one when we want to compare with
%the Sobolev space $W^{1,p}(X,\ud,\rho\mm)$. 
The second space we will consider is the subspace $H^{1,p}_\rho$ of $W^{1,p}_{\rho}$ defined by 
\begin{equation}\label{ii}
H^{1,p}_\rho:=\overline{{\rm{Lip}}(X)\cap W^{1,p}_\rho}^{\|\cdot\|_{\rho}}.
\end{equation} 

Even when the metric measure structure is Euclidean, it is well known that $H^{1,p}_\rho$ can be strictly included in $W^{1,p}_\rho$, see Section~\ref{sec:5}
for a more detailed discussion and examples. This gap suggests a discrepancy between the weighted spaces $H=W$ 
of the metric theory, obtained by considering $\rho\mm$ as reference measure, and the spaces $W^{1,p}_\rho$, $H^{1,p}_\rho$.

Our first main result states that if $\rho\in L^1_{{\rm loc}}(\mm)$, 
$\rho^{-1}\in L^{1/(p-1)}(\mm)$, $(X,\ud,\mm)$ is doubling and supports a $1$-Poincar\'e inequality
for Lipschitz functions, then
\begin{equation}\label{eq:may24}
W^{1,p}(X,\ud,\rho\mm)=H^{1,p}_\rho
\end{equation}
and the two spaces are isometric. Hence, the smaller of the two weighted Sobolev spaces can be naturally identified with the 
Sobolev space of the metric theory, with weighted reference measure $\rho\mm$.

In light of the equality $H=W$ and the Euclidean counterexamples to $H^{1,p}_\rho=W^{1,p}_\rho$, it is natural to expect that stronger
integrability properties of $\rho$ are needed to establish the equality $W^{1,p}(X,\ud,\rho\mm)=W^{1,p}_\rho$, namely 
\begin{equation}\label{eq:may23}
W^{1,p}(X,\ud,\rho\mm)=\left\{f\in W^{1,1}(X,\ud,\mm):\ |f|+|\nabla f|_w\in L^p(\rho\mm)\right\}.
\end{equation}
Notice that the inclusion $\subset$ readily follows by \eqref{eq:may24}.
Our second main result shows that \eqref{eq:may23} holds provided $(X,\ud,\mm)$ is doubling, supports a $1$-Poincar\'e inequality
for Lipschitz functions and $\rho$ satisfies the asymptotic condition
\begin{equation}\label{eq:may23bis}
\liminf_{n\to\infty}\frac{1}{n^p}\biggl(\int_X\rho^n\ud\mm\biggr)^{1/n}\biggl(\int_X\rho^{-n}\ud\mm\biggr)^{1/n}<\infty.
\end{equation}
This condition appeared first in the Euclidean context in \cite{Zhi}, dealing with $H^{1,2}_\rho=W^{1,2}_\rho$, 
see also the recent extension \cite{Surn} to any power $p>1$ and even to variable exponents.
As we illustrate below, the proof in \cite{Zhi} is sufficiently robust to be adapted, with minor variants,
to a nonsmooth context.

In view of the characterization in \cite{Sha}, we believe that \eqref{eq:may23} is conceptually interesting. Indeed, functions in the left hand side of \eqref{eq:may23}
are absolutely continuous (modulo the choice of an appropriate representative) along $\mathrm{Mod}_{p,\rho\mm}$-a.e. curve, while
functions in the right hand side are absolutely continuous along $\mathrm{Mod}_{1,\mm}$-a.e. curve. On the other hand, even with $p=1$,
it seems very difficult to connect the two notions of negligibility if $\rho$ and $\rho^{-1}$ are unbounded. As a matter of fact our proof
is very indirect and it would be nice to find a more direct explaination of the validity of \eqref{eq:may23}.%, maybe based on the fact 
%that under the doubling and Poincar\'e assumptions on the basic metric measure structure $(X,\ud,\mm)$ 
%``less'' curves are needed to provide the Sobolev regularity of a function (another instance of this phenomenon
%is given in Proposition~\ref{prop:propW11bis}).

We conclude the introduction by describing the structure of the paper. In Section~\ref{sec:2} we recall aspects of the theory of Sobolev spaces on metric measure spaces; we detail approximation results and the notion of measurable differentiable
structure from \cite{Che}.

In Section~\ref{sec:3} we introduce the weighted Sobolev spaces $W^{1,p}_\rho$ and $H^{1,p}_\rho$, showing
first completeness of $W^{1,p}_\rho$ under the assumption $\rho^{-1}\in L^{1/(p-1)}(\mm)$ and then reflexivity, under the additional
assumption that $(X,\ud,\mm)$ is doubling and satisfies a $1$-Poincar\'e inequality. The proof of reflexivity is particularly tricky and
it passes, as in \cite{Che} and \cite{AmbMar2}, through the construction of an equivalent uniformly convex norm. This involves a Lusin type approximation by Lipschitz functions. Notice this is not necessarily an approximation
in the norm of $W^{1,p}_\rho$,  since we know that additional assumptions on $\rho$ are needed to get density of Lipschitz functions, namely
the equality $W^{1,p}_\rho=H^{1,p}_\rho$. Then, using reflexivity and the $H=W$ theorem of the metric theory, we prove \eqref{eq:may24} in Theorem~\ref{thm:main1}.

Section~\ref{sec:4} is devoted to the proof of \eqref{eq:may23}, obtained in 
Theorem~\ref{mainthm} under the assumption \eqref{eq:may23bis}. Here we follow closely \cite{Zhi}, with
some minor adaptations due to the lack of differentiability of $f\mapsto\int |\nabla f|_w^p\ud\mm$ (potentially even for $p=2$).

In Section~\ref{sec:5} we recall an example from \cite{ChiSC}, showing that $H^{1,p}_\rho$ can be strictly included in $W^{1,p}_\rho$,
and we explore some variants of our results. In particular we relax the $1$-Poincar\'e assumption to a $p$-Poincar\'e assumption, 
modifying consequently the definitions of $W^{1,p}_\rho$ and $H^{1,p}_\rho$. Our main results still work, under the
$p$-Poincar\'e assumption, for these spaces and we prove that the new definitions coincide with \eqref{eq:Alica4} and \eqref{ii}
assuming the validity of the $1$-Poincar\'e inequality. Finally, we
discuss the notion of Muckenhoupt weight and the invariance of our assumptions on $(X,\ud,\mm)$ under the
replacement of $\mm$ by $\eta\mm$, with $\eta$ a Muckenhoupt weight.

\smallskip\noindent
{\bf Acknowledgements.} 
We would like to thank M. D. Surnachev for kindly explaining the change of weight in \cite{Zhi} (which we adapted in Proposition \ref{newweight}) and for informing us of his generalization of Zhikov's result to more general exponents \cite{Surn}. We thank J. Bj\"orn for pointing out a mistake in Proposition 5.2 in a previous version of the paper. The authors acknowledge the support of the grant ERC ADG GeMeThNES. The second author has been partially supported by the Gruppo
Nazionale per l'Analisi Matematica, la Probabilit\`a e le loro
Applicazioni (GNAMPA) of the Istituto Nazionale di Alta Matematica
(INdAM).

\section{Sobolev Spaces}\label{sec:2}

Throughout this paper we will denote by $(X, \ud)$ a complete separable metric space and by $\mm$ a locally finite (i.e. finite on bounded sets) Borel regular measure on $X$. In metric spaces Lipschitz functions play the role of smooth functions. We recall that a function $f\colon X \to \mathbb{R}$ is called Lipschitz if there exists $L\geq 0$ such that $|f(x)-f(y)|\leq L\ud(x,y)$ for all $x, \,y \in X$; we denote the smallest such constant $L$ by ${\rm{Lip}}(f)$ and denote the set of Lipschitz functions on $X$ by ${\rm{Lip}}(X)$. 

For a Lipschitz function $f$, a 
natural candidate for the modulus of gradient is given by the slope $|\nabla f|\colon X \to \mathbb{R}$, defined by
\begin{align*}
|\nabla f|(x):=\limsup_{y\to x} \frac{|f(y)-f(x)|}{\ud(y,x)}.
\end{align*}

%For Lipschitz functions on Euclidean spaces this slope agrees with the Euclidean norm of the gradient. We now recall a suitable definition of weak gradient and the corresponding Sobolev space \cite{Sha}. 
%Recall that one (equivalent) definition of Sobolev spaces is given by requiring that, for most lines, changes in the function can be estimated by a line integral of the weak gradient. This is the definition generalized by \cite{Sha}; in order to define this Sobolev space we first need a way to analyze families of curves.

\begin{defi}[Absolute continuity]
Let $J\subset \bbR$ be a closed interval and consider a curve $\gamma:J\to X$. We say that $\gamma$ is absolutely continuous if 
\begin{align}\label{AC}
\ud(\gamma(t),\gamma(s))\leq \int_s^t g(r)\ud r\quad \forall s,t\in J,\ s<t
\end{align}
for some $g\in L^1(J)$. 
\end{defi}

It is well known (see \cite[Proposition 4.4]{libro} for the proof) that every absolutely continuous curve $\gamma$ admits a minimal $g$ satisfying \eqref{AC}, called metric speed, denoted by $|\dot{\gamma}(t)|$ and given for a.e. $t\in J$ by
\[
|\dot{\gamma}(t)|=\lim_{s\to t}\frac{\ud(\gamma(s),\gamma(t))}{|s-t|}.
\]
We will denote by $C([0,1];X)$ the space of continuous curves from $[0,1]$ to $(X,\ud)$ endowed with the sup norm and by $AC([0,1],X)$ the subset of absolutely continuous curves. Using the metric derivative we can easily define curvilinear integrals, namely
$$
\int_\gamma g:=\int_0^1 g(\gamma(t))|\dot{\gamma}(t)|\ud t
$$
for all $g:X\to [0,\infty]$ Borel and $\gamma\in AC([0,1];X)$.

\begin{defi}[Modulus]\label{def:modulus}
Given $p\geq 1$ and $\Gamma\subset AC([0,1],X)$, the $p$-modulus $\mathrm{Mod}_{p,\mm}(\Gamma)$ is defined by
\[
\mathrm{Mod}_{p,\mm}(\Gamma):=\inf\Big\{\int_{X} h^p \ud \mm\ : \ \int_{\gamma} h\geq 1\ \forall \gamma\in \Gamma\Big\},
\] 
where the infimum is taken over all non-negative Borel functions $h:X\to [0,\infty]$.\\ We say that $\Gamma$ is $\mathrm{Mod}_{p,\mm}$-negligible 
if $\mathrm{Mod}_{p,\mm}(\Gamma)=0$.
\end{defi}

We can now give the definition of weak gradient and Sobolev space which we will use, see \cite{Sha}.

\begin{defi}[$p$-upper gradient]
For $p\geq 1$ we say that a Borel function $g:X\to [0,\infty]$ with $\int_{X} g^p\ud\mm<\infty$ is a 
$p$-weak upper gradient of $f$ if there exist a function $\widetilde f$ and a $\mathrm{Mod}_{p,m}$-negligible set $\Gamma$ 
such that $\widetilde f= f$\ $\mm$-a.e. in $X$ and
\begin{equation}\label{ugproperty}
|\widetilde f(\gamma_0)-\widetilde f(\gamma_1)|\leq \int_{\gamma} g\ \ud s\qquad\text{
for all $\gamma\in AC([0,1],X)\setminus \Gamma$.}
\end{equation}
\end{defi}

The following Theorem is classical, see \cite{ShaK,Sha} for a proof.

\begin{teo} \label{minlat}
For every $p\geq 1$ the collection of all $p$-weak upper gradients of a map $f:X\to \bbR$ is a closed convex lattice in $L^p(\mm)$.\\
Moreover, if the collection of all $p$-weak upper gradients of $f$ is nonempty then it contains a unique element of smallest $L^p(\mm)$ norm.
We shall denote it by $|\nabla f|_{p,\mm}$. 
\end{teo}

From now on we denote the $1$-weak gradient of $f$ with respect to $\mm$ simply by $|\nabla f|_{w}$, so $|\nabla f|_{w}=|\nabla f|_{1,\mm}$. Following \cite{Sha} we can now define the Sobolev space from which we will define weighted Sobolev spaces on metric measure spaces.

\begin{defi}[Sobolev space $W^{1,p}(X,\ud,\mm)$]\label{defsob}
For each $p\geq 1$ we define $W^{1,p}(X,\ud,\mm)$ to be the Banach space of ($\mm$-a.e. equivalence classes of) functions $f\in L^{p}(\mm)$ having a 
$p$-weak upper gradient, endowed with the norm 
\[
\|f\|_{W^{1,p}(\mm)}^p:=\int_{X} |f|^p \ud \mm+\int_{X} |\nabla f|^p_{p,\mm}\  \ud \mm.
\]
\end{defi}

It can be proved (see \cite[Proposition 5.3.25]{ShaK}) that $|\nabla f|_{p,\mm}$ is local, namely
\begin{equation}\label{eq:loca}
|\nabla f|_{p,\mm}=|\nabla g|_{p,\mm}\qquad \mm\mbox{-a.e. on}\ \{f=g\}
\end{equation}
for all $f,\,g\in W^{1,p}(X,\ud,\mm)$.

Definition~\ref{defsob} is by now classical and it goes back to the pioneering work \cite{Sha}, where the author also proved that if $p>1$ then the space $W^{1,p}(\mm)$ coincides with the Sobolev space defined by Cheeger \cite{Che} in terms of approximation by pairs $(f_n,g_n)$, with $f_n\to f$ in $L^p(\mm)$, $g_n$ an upper gradient of $f_n$ and $\{g_n\}_{n\in\bbN}$
bounded in $L^p(\mm)$. More recently, the first named author, Gigli and Savar\'e (see \cite{AGS12} for $p=2$ and \cite{AmbGiSav} for $p>1$) improved
this equivalence result proving existence of an approximation by Lipschitz functions, with slopes (or even asymptotic
Lipschitz constants, see \cite{AmbMar}) as upper gradients. More precisely, defining
\begin{eqnarray}
H^{1,p}(X,\ud,\mm)&=&\bigl\{f\in L^p(\mm):\ \text{$\exists$\, $f_n\in \mathrm{Lip}(X)\cap L^p(\mm)$} \\
 && \text{with
$f_n\to f$ in $L^p(\mm)$ and $\sup_n\int_X|\nabla f_n|^p\ud\mm<\infty$}\bigr\}\nonumber,
\end{eqnarray}
the following result holds.

\begin{teo} \label{thH=W}
$W^{1,p}(X,\ud,\mm)=H^{1,p}(X,\ud,\mm)$ for all $p>1$. In addition, for all functions $f\in W^{1,p}(X,\ud,\mm)$ the following holds:
\begin{equation}\label{eq:approximate}
\text{there exist $f_n\in \mathrm{Lip}(X)\cap L^p(\mm)$ with $f_n\to f$ and $|\nabla f_n|\to |\nabla f|_{p,\mm}$ in $L^p(\mm)$.}
\end{equation}
\end{teo}

On the contrary, the picture for $p=1$ is far from being complete since at least three definitions are available 
(see also \cite{AmbMar2} and the forthcoming paper \cite{AmbPiSpe} for a discussion on this subject). 

For our analysis of weighted Sobolev spaces we will require that the measure $\mm$ is doubling and that a $p$-Poincar\'e inequality holds; we recall these properties now. Doubling metric measure spaces which satisfy a $p$-Poincar\'e inequality are often called $\PI_p$ spaces and there are many known examples which differ from ordinary Euclidean spaces, see for instance \cite{HaK,Hei,Hei2}.

\begin{defi}[Doubling]\label{defi:doub}
A locally finite Borel measure $\mm$ on $(X,\ud)$ is doubling if it gives finite positive measure to balls and there exists a constant $C>0$ such that
\begin{align}\label{doub}
\mm(B(x,2r))\leq C\mm(B(x,r))\quad \forall x\in X,\ r>0.
\end{align}
In this case we also say that the metric measure space $(X,\ud,\mm)$ is doubling.
\end{defi}

A metric measure space gains additional structure if a Poincar\'e inequality is imposed; this type of inequality is a control on the local variation of 
a Lipschitz function using infinitesimal behaviour encoded by the slope.

\begin{defi}[$p$-Poincar\'e]\label{defi:poinc}
For $p\in [1,\infty)$, we say that a $p$-Poincar\'e inequality holds for Lipschitz functions if there exist 
constants $\tau,\,\Lambda>0$ such that for every $f\in {\rm{Lip}}(X)$ and for every $x\in \supp(\mm)$, $r>0$ the following inequality holds:
\begin{align}\label{poinc}
\ave_{B(x,r)}|f-f_{B(x,r)}|\ud \mm\leq \tau r \biggl(\ave_{B(x,\Lambda r)} |\nabla f|^p \ud \mm\biggr)^{1/p},
\end{align}
where, here and in the sequel,
\[ f_{A}=\ave_{A} f \ud \mm:=\frac{1}{\mm(A)} \int_{A} f \ud \mm.\]
We will say that a constant is structural if it depends only on the doubling constant in \eqref{doub} and the constants $\tau,\,\Lambda$ in \eqref{poinc}.
\end{defi}

Notice that, by the H\"older inequality, the $\PI_p$ condition becomes weaker as $p$ increases, so $\PI_1$ is the strongest
possible assumption. A remarkable result (see \cite{KZ}) is that $\PI_p$ is an open ended condition
in $(1,\infty)$, namely $\PI_p$ for some $p\in (1,\infty)$ implies $\PI_q$ for some exponent $q\in (1,p)$.

Thanks to \eqref{eq:approximate}, for all $p>1$ and $f\in W^{1,p}(X,\ud,\mm)$, under the $\PI_p$ assumption it holds 
\begin{align}\label{weakPoincW}
\ave_{B(x,r)}|f-f_{B(x,r)}|\ud\mm \leq \tau r \biggl(\ave_{B(x,\Lambda r)} |\nabla f|^p_{p,\mm}\ud\mm\biggr)^{1/p}
\end{align}
for all $x\in \supp(\mm)$ and $r>0$, where $\tau,\,\Lambda$ are as in \eqref{poinc} (see also \cite[Theorem 2]{Keith2}). The inequality is still valid with $p=1$ under
the $\PI_1$ assumption. Indeed, the space $W^{1,1}(X,\ud,\mm)$ is contained in the space $BV(X,\ud,\mm)$ of functions
having a ``measure'' upper gradient considered in \cite{AmbMar2}. The equivalence result with Miranda's definition of $BV$ 
(which parallels the ideas provided in \cite{AGS12} for the Sobolev spaces) provided in \cite{AmbMar2} ensures the
existence of sequence $(f_n)\subset
\mathrm{Lip}(X)\cap L^1(\mm)$ with $f_n\to f$ in $L^1(\mm)$ and $|\nabla f_n|\mm\to |Df|$ weakly as measures. Taking the limit,
one obtains \eqref{weakPoincW} with $|Df|(\overline{B}(x,\Lambda r))$ in the right hand side. In the case
when $f\in W^{1,1}(X,\ud,\mm)\subset BV(X,\ud,\mm)$ one can use the inequality $|Df|\leq |\nabla f|_{1,\mm}\mm$ 
to conclude, see also \cite{AmbPiSpe} for a more detailed discussion.

We now recall the relevant properties of the maximal operator.

\begin{defi} [Maximal operator] Given a locally integrable Borel function $f:X\to \bbR$, we define the maximal function 
$Mf\colon X \to [0,\infty]$ associated to $f$ by
\begin{equation}\label{eq:def_maximal}
M f(x):= \sup_{r>0}\ave_{B(x,r)} |f|\ud\mm .
\end{equation}
\end{defi}

Since $\mm$ is doubling we know \cite[Theorem 2.2]{Hei} that for $q>1$ the maximal operator is a bounded linear map from $L^{q}(\mm)$ to $L^{q}(\mm)$; 
more precisely, there exists a constant $C>0$ depending only on the doubling constant such that
\[\|M(f)\|_{L^{q}(\mm)}\leq \frac{C}{(q-1)^{1/q}}\|f\|_{L^{q}(\mm)}\]
for all $f\in {L^{q}(\mm)}$. For $q=1$ the maximal operator is also weakly bounded, namely
\begin{equation}\label{eq:maximal_estimates1}
\sup_{\lambda>0}\lambda\mm\bigl(\{M(g)>\lambda\}\bigr)\leq\int_X|g|\ud\mm.
\end{equation}
We will also need the asymptotic estimate
\begin{equation}\label{eq:maximal_estimates2}
\lim_{\lambda\to\infty}\lambda \mm\bigl(\{M(g)>\lambda\}\bigr)=0.
\end{equation}
This asymptotic version follows by \eqref{eq:maximal_estimates1}, taking the inclusion 
$$\{M(g)>2\lambda\}\subset \{M((|g|-\lambda)^+)>\lambda)\}$$ into account.

Recall that $x$ is a Lebesgue point of a locally integrable function $u$ if
$$
\lim_{r\downarrow 0}\ave_{B(x,r)}|u(y)-u(x)|\ud\mm(y)=0.
$$
This notion is sensitive to modification of $u$ in $\mm$-negligible sets. With a slight abuse of notation we shall also apply this notion to 
Sobolev functions, meaning that we have chosen a representative in the equivalence class.

We now state a key approximation property for functions in $W^{1,p}(X,\ud,\mm)$, valid under the doubling and $p$-Poincar\'e
assumptions. We give a sketch of proof for the reader's convenience, but these facts are well known, see for instance \cite{Che}, \cite{Sha}
or the more recent paper \cite{AmbMar} where some proofs are revisited.

\begin{prop}\label{prop:propW11}
Assume that $p\in [1,\infty)$ and that $(X,\ud,\mm)$ is a $\PI_p$ metric measure space. Then,
for all $f\in W^{1,p}(X,\ud,\mm)$ there exist $f_n\in \mathrm{Lip}(X)\cap W^{1,p}(X,\ud,\mm)$ and Borel sets $E_n$
with:
\begin{itemize}
\item[(i)] $E_n\subset E_{n+1}$ and $\sup_n n^p\mm(X\setminus E_n)<\infty$, so that $\mm(X\setminus\cup_n E_n)=0$,
\item[(ii)] $|f_n|\leq n$, $\mathrm{Lip}(f_n)\leq Cn$, $f=f_n$ $\mm$-a.e. in $E_n$,
\item[(iii)] $|f_n-f|\to 0$ and $|\nabla(f_n-f)|_{p,\mm}\to 0$ in $L^p(\mm)$. 
\end{itemize}
Furthermore, there is a structural constant $c$ such that for all $f\in \mathrm{Lip}(X)\cap W^{1,p}(X,\ud,\mm)$,
\begin{equation}\label{eq:may242}
|\nabla f|_{p,\mm}\leq |\nabla f|\leq c|\nabla f|_{p,\mm}\qquad\text{$\mm$-a.e. in $X$}.
\end{equation}
\end{prop}

\begin{proof} Recall the definition of the maximal operator $M$ w.r.t. $\mm$ from \eqref{eq:def_maximal}. By iterating the estimate \eqref{weakPoincW} on concentric balls (see for instance \cite{Che} or
\cite[Lemma~8.2]{AmbMar}) one can obtain the estimate
\begin{equation}\label{eq:localip}
|f(x)-f(y)|\leq C \ud(x,y) (M^{1/p}(|\nabla f|_{p,\mm}^p)(x)+M^{1/p}(|\nabla f|_{p,\mm}^p)(y))
\end{equation}
whenever $f\in W^{1,p}(X,\ud,\mm)$ and $x,\,y\in X$ are Lebesgue points of (a representative of) $f$. Set now  
\[g:=\max \{|f|,M^{1/p}(|\nabla f|^p_{p,\mm})\}.\]
We also define 
\[E_n:=\left\{x\in X:\ \text{$x$ is a Lebesgue point of $f$ and $g(x)\leq n$}\right\}.\]
Notice that since $\{g>n\}$ is contained in $\{|f|>n\}\cup\{ M(|\nabla f|^p_{p,\mm})>n^p\}$,
the set $X\setminus E_n$ has finite $\mm$-measure, and more precisely Markov inequality
and the weak maximal estimate give that $\sup_n n^p\mm(X\setminus E_n)<\infty$.

Using \eqref{eq:localip} and the definition of $g$ we obtain
\[|f(x)-f(y)|\leq Cn \ud(x,y),\qquad |f(x)|\leq n\]
for all $x, \,y \in E_n$. By the McShane lemma (see for instance \cite{Hei} for the simple proof) 
we can extend $f|_{E_{n}}$ to a Lipschitz function $f_n$ on $X$ preserving the Lipschitz constant and
the sup estimate, namely ${\rm Lip}(f_n)\leq Cn$ and
$|f_n|\leq n$. We claim that $|\nabla f_n|_{p,\mm}\leq Cn$ $\mm$-a.e. in $X$. Indeed, since
$|\nabla f_n|_{p,\mm}\leq |\nabla f_n|$ $\mm$-a.e in $X$, we get 
\[
|\nabla f_n|_{p,\mm}\leq |\nabla f_n|\leq \mathrm{Lip}(f_n)\leq Cn
\qquad\text{$\mm$-a.e. in $X$.}
\]
Furthermore, by the locality of the weak gradient we obtain 
$|\nabla (f_n-f)|_{p,\mm}=0$ $\mm$-a.e. on the set $E_n$. 

Using these facts and \eqref{eq:maximal_estimates2} it is straightforward to check, by dominated convergence, that
$f_n\to f$ in $L^p(\mm)$ and that $|\nabla (f_n-f)|_{p,\mm}\to 0$ in $L^p(\mm)$.
 
The proof of \eqref{eq:may242} relies on a localized version of \eqref{eq:localip}, namely
\begin{equation}\label{eq:localiplip}
|f(x)-f(y)|\leq C \ud(x,y) (M_{2\Lambda r}^{1/p}(|\nabla f|_{p,\mm}^p)(x)+M_{2\Lambda r}^{1/p}(|\nabla f|_{p,\mm}^p)(y))
\end{equation}
for all Lebesgue points $x,\,y\in X$ of $f$ with $\ud(x,y)<r$, where $M_s$ is the maximal operator
on scale $s$ (i.e. the supremum in \eqref{eq:def_maximal} is restricted to balls with radius smaller than $s$).
The idea of the proof is to differentiate at Lebesgue points $x$
of $|\nabla f|_{p,\mm}^p$, letting eventually $r\downarrow 0$ and using the fact that
$M_r(g)\downarrow |g|$ at Lebesgue points of $g$ as $r\downarrow 0$,
see \cite{Che} or \cite[Proposition~47]{AmbMar} for details.
\end{proof}

With a similar proof, using the boundedness of the maximal operator, one can prove the following proposition
(see \cite{AGS12}, \cite{DmSp} for counterexamples showing that the $\PI_q$ assumption can not be removed).

\begin{prop}\label{prop:propW11bis}
Let $(X,\ud,\mm)$ be a $\PI_q$ metric measure space and $p>q$. If it holds that $f~\in~W^{1,q}(X,\ud,\mm)$ and both $f$ and $|\nabla f|_{q,\mm}$ belong to
$L^p(\mm)$, then $f\in W^{1,p}(X,\ud,\mm)$. 
\end{prop}

In the sequel we will use the fact that $\PI_p$ spaces for some $p\geq 1$ admit a differentiable structure; we conclude this section by recalling some aspects of Cheeger's remarkable theory \cite{Che, KleMc}, which provides a differentiable structure that will play a role in the reflexivity of the weighted Sobolev spaces.

\begin{defi}\label{defi:diffstructure}
A measurable differentiable structure on a metric measure space $(X,\ud,\mm)$ is a countable collection of pairs $\{(U_{\alpha},\varphi_{\alpha})\}$, called local charts, that satisfy the following conditions:
\begin{enumerate}
	\item[(i)] Each $U_{\alpha}$ is a measurable subset of $X$ with positive measure, and $\mm(X\setminus \bigcup_{\alpha} U_{\alpha})=0$.
	\item[(ii)] Each $\varphi_{\alpha}$ is a Lipschitz map from $X$ to $\mathbb{R}^{N(\alpha)}$ for some integer $N(\alpha)\geq 1$, and moreover 
	$N:=\sup_\alpha N(\alpha)<\infty$.
	\item[(iii)] For every $f \in \mathrm{Lip}(X)$ and for every $\alpha$ there exists an $\mm$-measurable function $\ud^{\alpha} f:U_{\alpha}\to \bbR^{N(\alpha)}$ such that 
	\[
	\limsup_{y\to x}\frac{|f(y)-f(x)-\ud^{\alpha} f(x)\cdot (\varphi_{\alpha}(y)-\varphi_{\alpha}(x))|}{\ud(x,y)}=0\qquad\text{for $\mm$-a.e. $x\in X$}
	\]
and $\ud^{\alpha} f$ is unique up to $\mm$-negligible sets.
\end{enumerate}
\end{defi}

The following theorem is proved in \cite{Che}, here we state it in the form needed in this paper.

\begin{teo}[Existence of a measurable differentiable structure]\label{teo:struCheg}
If $(X,\ud,\mm)$ is a $\PI_p$ metric measure space for some $p\geq 1$, then $X$ admits a measurable differentiable structure and the integer $N$ in 
Definition~\ref{defi:diffstructure}(ii) depends only on the structural constants. 
Moreover, for all $\alpha$ and $\mm$-a.e. $x \in U_{\alpha}$, there is a Hilbertian norm $\|\cdot\|_x$ on $\mathbb{R}^{N(\alpha)}$ such that
$x\mapsto \|d^{\alpha}f(x)\|_x$ is $\mm$-measurable in $U_\alpha$ and
\begin{equation}
\|d^{\alpha} f(x)\|_x\leq |\nabla f|(x)\leq M\|d^{\alpha}f(x)\|_x\qquad\text{for $\mm$-a.e. $x\in U_{\alpha}$, for all $f\in {\rm Lip}(X)$,}
\end{equation}
where $M>0$ is a constant independent of $\alpha$.
\end{teo}

\section{Weighted Sobolev Spaces}\label{sec:3}

In this section we will define weighted Sobolev spaces and prove that, under natural integrability assumptions on the weight, we obtain a reflexive Banach space. Recall the $1$-weak upper gradient of $f$ with respect to $\mm$ is denoted by $|\nabla f|_{w}$, so $|\nabla f|_{w}=|\nabla f|_{1,\mm}$.

\begin{defi}[Weighted space $W^{1,p}_\rho(X,\ud,\mm)$]
Let $p>1$ and let $\rho:X\to [0,\infty]$ be a Borel function satisfying $\rho^{-1}\in L^{{1}/{(p-1)}}(\mm)$; we define the weighted Sobolev 
space $W^{1,p}_\rho(X,d,\mm)$ by
\begin{equation}\label{eq:newlabel}
W^{1,p}_{\rho}(\mm):=\Big\{f\in W^{1,1}(X,\ud,\mm) : \int_X|f|^p\ \rho \ud\mm+ \int_{X} |\nabla f|^p_{w}\ \rho \ud \mm<\infty\Big\}.
\end{equation}
We endow $W^{1,p}_\rho(X,\ud,\mm)$, shortened to $W^{1,p}_{\rho}$, with the norm:
\[\| f \|_{\rho}^p:=\int_X |f|^p\ \rho \ud\mm+\int_{X} |\nabla f|^p_{w}\  \rho\ud \mm.\]
\end{defi}

\begin{oss}{\rm 
The fact that $\|\cdot\|_{\rho}$ is a norm is a consequence of the observation $\rho>0$ $\mm$-a.e. in $X$ and of
the following elementary properties 
\begin{align*}
|\nabla (f+g)|_{w} &\leq |\nabla f|_{w}+|\nabla g|_{w}& \, &\mm\mbox{-a.e. in}\ X,\\
|\nabla (\lambda f)|_{w} &= |\lambda| |\nabla f|_{w}& \, &\mm\mbox{-a.e. in}\ X,\, \forall \lambda\in\bbR.
\end{align*}}
\end{oss}

Note that, using H\"older's inequality, it follows that
\begin{equation}\label{eq:embedding}
\|f\|_{W^{1,1}}\leq 2\left( \int \rho^{-\frac{1}{p-1}} \ud \mm \right)^{\frac{p-1}{p}} \|f\|_{\rho}.\end{equation}
That is, $W^{1,p}_\rho(\mm)$ embeds continuously into $W^{1,1}(X,\ud,\mm)$ provided $\rho^{-1}\in L^{1/(p-1)}(\mm)$; a similar
calculation shows that the definition of $W^{1,p}_\rho(\mm)$ would be unchanged if we replace $W^{1,1}(X,\ud,\mm)$ by
$W^{1,1}_{\rm loc}(X,\ud,\mm)$ (namely the space of functions whose $1$-weak upper gradient is integrable on bounded sets)
in \eqref{eq:newlabel}.

We use the embedding of $W^{1,p}_\rho(\mm)$ to prove completeness of the weighted Sobolev space, building on the completeness of $W^{1,1}(X,\ud,\mm)$.

\begin{prop}[Completeness of $W^{1,p}_\rho$]
For every $p>1$, the weighted Sobolev space $(W^{1,p}_{\rho},\|\cdot\|_{\rho})$ is a Banach space whenever
$\rho^{-1}\in L^{1/(p-1)}(\mm)$.
\end{prop}

\begin{proof} Suppose that $(f_{i})_{i=1}^{\infty}$ is a Cauchy sequence in $W_{\rho}^{1,p}$, and let $\omega_i \downarrow 0$ be such that $\|f_n-f_m\|_{\rho}^p\leq\omega_i$ whenever $n,\,m\geq i$. From \eqref{eq:embedding} and the
completeness of $W^{1,1}(X,\ud,\mm)$ we obtain that there exists $f\in W^{1,1}(X,\ud,\mm)$ such that $f_n\to f$ in $W^{1,1}(X,\ud,\mm)$ and hence, up to a subsequence, we can also assume that $f_n\to f$ pointwise $\mm-$a.e. Since $L^p(\rho \mm)$ is a Banach space and $L^p(\rho\mm)$ convergence implies
the existence of subsequences convergent $\rho\mm-$a.e. in $X$, we deduce $f$ is also the limit of $f_n$ in $L^p(\rho \mm)$. As $|\nabla (f_n-f)|_w\to 0$ in $L^1(\mm)$ implies
 $|\nabla (f_n-f_m)|_w\to |\nabla (f-f_m)|_w$ in $L^1(\mm)$, we can pass to the limit
 in $\||\nabla (f_n-f_m)|_w\|^p_{L^p(\rho\mm)}\leq \omega_i$  to obtain
$$
\int_X |\nabla (f-f_m)|_w^p\rho\ud\mm\leq\omega_i\qquad\forall m\geq i.
$$
This proves that $\int_X|\nabla f|_w^p\rho\ud\mm<\infty$, so that $f\in W^{1,p}_\rho$, and that
$f_m\to f$ in $W^{1,p}_\rho$.
\end{proof}

\begin{teo}[Reflexivity of $W^{1,p}_\rho$]\label{reflexivity}
Suppose $(X,\ud,\mm)$ is a $\PI_1$ metric measure space, $\rho\in L^1_{{\rm loc}}(\mm)$ and
$\rho^{-1}\in L^{{1}/{(p-1)}}(\mm)$.
Then $(W^{1,p}_{\rho},\|\cdot\|_{\rho})$ is reflexive for all $p>1$.
\end{teo}

\begin{proof}
By assumption $(X,\ud,\mm)$ is a $PI_{1}$ space; therefore, by Theorem \ref{teo:struCheg}, it admits a differentiable structure consisting of charts $(U_{\alpha}, \varphi_{\alpha})$, where $U_{\alpha}\subset X$ is measurable and $\varphi_{\alpha}\colon X \to \mathbb{R}^{N(\alpha)}$ are Lipschitz with $N(\alpha)\leq N$, with respect to which Lipschitz functions are differentiable. Further, for all $f\in {\rm Lip}(X)$,
\begin{equation}\label{equivalence}
\|\ud^{\alpha} f(x)\|_x\leq |\nabla f|(x)\leq M\|\ud^{\alpha} f(x)\|_x\qquad\text{for $\mm$-a.e. $x\in U_{\alpha}$}
\end{equation}
where $\|\cdot\|_{x}$ is an inner product norm on $\mathbb{R}^{N(\alpha)}$ for $x \in U_{\alpha}$ and $M$ is a positive constant depending only on $N$.

Without loss of generality we assume the sets $U_{\alpha}$ are disjoint and denote derivatives by $\ud f(x)$ instead of $\ud^{\alpha} f(x)$ when $x \in U_{\alpha}$. We now observe that we can also assume $N(\alpha)=N$ for all $\alpha$ by replacing:
\begin{itemize}
\item The inner product norm $\|\cdot\|_{x}$ on $\mathbb{R}^{N(\alpha)}$ by the semi inner product norm $\|p(\cdot)\|_{x}$ on $\mathbb{R}^{N}$, where $p\colon \mathbb{R}^{N} \to \mathbb{R}^{N(\alpha)}$ is the projection onto the first $N(\alpha)$ coordinates for $x \in U_{\alpha}$. Here a semi inner product $\langle \cdot,\cdot \rangle$ satisfies the usual properties of an inner product except for positive definiteness - the corresponding semi inner product norm is then given by $\|v\|^{2}=\langle v, v\rangle$.
\item The derivative $\ud f(x) \in \mathbb{R}^{N(\alpha)}$ by $(\ud f(x),0) \in \mathbb{R}^{N}$ for $x \in U_{\alpha}$.
\end{itemize}
After this replacement the map $\ud$ still satisfies the equivalence \eqref{equivalence}. Clearly $\ud$ is a linear map from ${\mathrm{Lip}}(X)$ 
to the space of ($\mm$-a.e. defined) $\mathbb{R}^{N}$ valued measurable functions on $X$. We split the the proof of reflexivity into three steps. Note that the first step is known and follows from results of \cite{FHK}, but we present it explicitly for the sake of readability.

\smallskip
\textbf{Step 1.} \textit{\textit{We construct an equivalent norm on $W^{1,p}_{\rho}$.}}
\smallskip

We first define a (non linear) map $D$ from $W^{1,p}_{\rho}$ to non negative ($\mm$-a.e. defined) 
measurable functions on $X$, which we denote by $D_{x}(g)$ instead of $D(g)(x)$, satisfying:
\begin{align}\label{propN} 
D_{x}(\lambda g)&=|\lambda| D_{x}(g) &&\mm\mbox{-a.e.}\  x\in X, \ \forall g \in W^{1,p}_{\rho},\ \lambda \in \mathbb{R},\nonumber\\
D_{x}(g+h)&\leq D_{x}(g)+D_{x}(h) &&\mm\mbox{-a.e.}\  x\in X,\ \forall g, h\in W^{1,p}_{\rho},\\
\nonumber
D_{x}(f)&=\|\ud f(x)\|_x  &&\mm\mbox{-a.e.}\  x\in X,\ \forall f\in {\rm{Lip}}(X)\cap W^{1,p}_{\rho},\\
\label{estimate} D_{x}(g)\leq &|\nabla g|_{w}(x)\leq MD_{x}(g) &&\mm\mbox{-a.e. } x\in X,\  \forall g\in W^{1,p}_{\rho}.
\end{align}
Fix $g\in W^{1,p}_{\rho}\subset W^{1,1}$. By Proposition~\ref{prop:propW11}, there exists a sequence of Lipschitz functions $g_{n}$ such that $|\nabla(g_{n}-g)|_{w}\to 0$ in $L^{1}(\mm)$ and
\[\sum_{n=1}^{\infty} \mm \{g\neq g_{n}\}<\infty.\]
Let $A_{n}:=\{g=g_{n}\}$ and $G_{T}:=\cap_{n \geq T}A_{n}$; by the Borel-Cantelli lemma it follows that  $\cup_{T}G_{T}$ has full $\mm$-measure. For $n,\,m>T$ we have, by Proposition~\ref{prop:propW11}(i), $|\nabla (g_{n}-g_{m})|_{w}=0$ $\mm$-a.e. in $G_{T}$. Hence, by inequality \eqref{eq:may242}, $n, \,m>T$ implies $|\nabla (g_{n}-g_{m})|=0$ $\mm$-a.e. in $G_{T}$. We now claim that, for each fixed $x\in G_{T}$, $\|\ud g_{n}(x)\|_{x}$ is constant as a function of $n>T$. 
Indeed, for $n, \,m>T$, using \eqref{equivalence},
\begin{align*}
|\|dg_{n}(x)\|_{x}-\|dg_{m}(x)\|_{x}|&\leq \|d(g_{n}-g_{m})(x)\|_{x}\\
&\leq |\nabla (g_{n}-g_{m})|(x)\\
&=0.
\end{align*}
We define $D_{x}(g):=\|\ud g_{n}(x)\|_{x}$ for $x \in G_{T}$ and $n>T$. It is easy to show that if we took a different sequence of Lipschitz functions $\widetilde{g}_{n}$ with $\sum_{n=1}^{\infty} \mm \{g\neq \widetilde{g}_{n}\}<\infty$ then we obtain the same definition of $D_{x}(g)$ up to $\mm$-a.e. equivalence. 

Using the measurability of the differential map we easily obtain the measurability of $x\mapsto D_{x}(g)$. Clearly for every $f\in W^{1,p}_{\rho}$ we have $D_x(f)\geq 0$. For every $f,g\in \mathrm{Lip}(X)$ and $\lambda\in\bbR$ we know
\[\ud(f+g)=\ud(f)+\ud(g) \text{ and } \ud(\lambda f)=\lambda\ud(f)\qquad\text{$\mm$-a.e. in $X$.}\]
This implies \eqref{propN}, since approximations for $f$ and $g$ give rise to approximations for $f+g$ and $\lambda f$. In order to prove \eqref{estimate}, we first remark that, by Proposition~\ref{prop:propW11}(iii) and \eqref{estimate}, we get 
\begin{align}\label{estimate2}
\|\ud g_n(x)\|_x\leq |\nabla g_n|_{w}\leq M \|\ud g_n(x)\|_x.
\end{align}
Since $|\nabla g_{n}|_{w} \to |\nabla g|_{w}$ in $L^{1}(\mm)$, it follows, up to a subsequence, $|\nabla g_{n}|_{w}(x) \to |\nabla g|_{w}(x)$ for $\mm$-a.e.\  $x \in X$. Therefore, the conclusion follows by letting $n\to \infty$ in \eqref{estimate2} and recalling that $\|\ud g_n(x)\|_x$ is constant for large $n$.

Using \eqref{propN} and \eqref{estimate} it is easy to see that the following expression defines an equivalent norm on $W^{1,p}_\rho$:
\[\|f\|_{\mathrm{Ch},\rho}:=\biggl(\int_X (|f(x)|^{p}+(D_{x}(f))^{p}) \rho(x)\,\ud\mm(x)\biggr)^{1/p}. \]

\smallskip
\textbf{Step 2.} \textit{Suppose $Q\colon X\to [0,\infty)$ is a $\mm$-measurable function; then the seminorm
\[f\mapsto \|f\|_{\mathrm{Ch},Q}= \biggl(\int_X (|f(x)|^{p}+(D_{x}(f))^{p}) Q(x)\,\ud \mm(x)\biggr)^{1/p}\]
is uniformly convex on the intersection ${\rm Lip}(X)\cap W^{1,p}_\rho$, with modulus of convexity independent of $Q$.}
\smallskip

Suppose $(Y,\mathcal{F})$ is a measurable space and, for each $y \in Y$, $\mathbb{R}^{n}$ is equipped with a semi inner product norm $\|\cdot\|_{y}$ such that $y\mapsto \|v\|_{y}$ is measurable for any $v \in \mathbb{R}^{n}$. By polarization, the map $y\mapsto \langle v,w\rangle_y$ is measurable for any $v,\,w\in \mathbb{R}^{n}$, where $\langle\cdot,\cdot \rangle_y$ denotes the induced semi inner product. Representing the semi inner product by a symmetric positive semidefinite matrix $y\mapsto A_y$ in the canonical basis of $\bbR^n$, it is clear that the entries of $A_y$ are also measurable.

It is well known (see for instance \cite{Bat}) that for any symmetric positive semidefinite matrix $A$ there exists a unique 
symmetric matrix $\sqrt{A}$ such that $\sqrt{A}\sqrt{A}=A$; in addition, the map $A\mapsto \sqrt{A}$ is continuous. As the composition of a continuous and a measurable map, the entries of $\sqrt{A_{y}}$ are measurable. Further, we can write $\langle v, w \rangle_{y} = (\sqrt{A_{y}}v)^{t}(\sqrt{A}_{y}w)$, which implies $\|v\|_{y}=|\sqrt{A_{y}}v|$ for all $v\in \mathbb{R}^{n}$, where $|\cdot |$ denotes the Euclidean norm on $\mathbb{R}^{n}$.

Using the discussion above, for each $x\in X$, we choose an $N \times N$ matrix $B_{x}$ such that $\|v\|_{x}=|B_{x}v|$ for all $v\in \mathbb{R}^{N}$.  Let $X_{1}$ and $X_{2}$ be two disjoint copies of $X$ supporting copies $\mm_{1}$ and $\mm_{2}$ of $\mm$ respectively. We define the following function $\widehat{f}\colon X_{1}\cup X_{2}\to \mathbb{R}^{N}$,
\begin{align*}
\widehat{f}(x):=\bigg \{
\begin{array}{rl}
(Q(x)^{1/p}f(x),0,\ldots,0) &x \in X_{1}\\
Q(x)^{1/p}B_{x}df(x)  &x \in X_{2}.\\
\end{array}
\end{align*}
Clearly $\widehat{f}$ is measurable. By using the equality $D_{x}(f)=\|\ud f(x)\|_x$ $\mm$-a.e. in $X$ it is simple to verify $\|\widehat{f}\|_{L^{p}(m_{1}+m_{2})}=\|f\|_{\mathrm{Ch},Q}$. Since the transformation $f\mapsto \widehat{f}$ is linear the usual uniform convexity of $L^{p}$ spaces implies uniform convexity of the norm $\|\cdot\|_{\mathrm{Ch},Q}$ on ${\rm Lip}(X)\cap W^{1,p}_\rho$ (with modulus independent of $Q$).

\smallskip
\textbf{Step 3.} \textit{The norm $\|\cdot\|_{\mathrm{Ch},\rho}$ is uniformly convex on $W^{1,p}_\rho$.}
\smallskip

It is an easy consequence of \eqref{propN}, \eqref{estimate} and locality for the weak gradient that if $g, h \in W^{1,p}_{\rho}$ and $g=h$ on a measurable set $E$ then $D_{x}g=D_{x}h$ for $\mm$-a.e. $x \in E$. We use this locality property throughout the sequel.

Given $\epsilon\in (0,1)$, let $\delta=\delta(\epsilon)\in (0,1)$ be given by the uniform convexity proved in the previous step. Suppose $f,\,g\in W^{1,p}_\rho$ satisfy $\|f\|_{\mathrm{Ch},\rho}=\|g\|_{\mathrm{Ch},\rho}=1$ and $\|f-g\|_{\mathrm{Ch},\rho}\geq \varepsilon$. Using Proposition~\ref{prop:propW11}, we can find an increasing family of 
bounded sets $E_n$ such that $\mm(X\setminus \cup_{n}E_{n})=0$ on which $f|_{E_n}$ and $g|_{E_n}$ are Lipschitz. Set $\rho_n(x):=\rho(x)\chi_{E_n}(x)$, where $\chi_{E_n}$ is the characteristic function of $E_n$. We first extend $f|_{E_n}$ and $g|_{E_n}$ to Lipschitz functions $f_n$ and $g_n$ on $X$ with bounded support. An easy argument using locality of the weak gradient and local integrability of $\rho$ shows that $f_{n}, g_{n} \in {\rm Lip}(X)\cap W^{1,p}_\rho$.

Next, let $\widetilde{f}_{n}:=f_{n}/z_{n}$ and $\widetilde{g}_{n}:=g_{n}/w_{n}$ for some scalars $z_{n}$ and $w_{n}$ such that
\[\|\widetilde{f}_{n}\|_{\mathrm{Ch},\rho_{n}}^{p}=\int_X (|{\widetilde{f}_n}(x)|^{p}+(D_{x}(\widetilde{f}_n))^{p}) \rho_{n}(x)\,\ud\mm(x)=1\]
and
\[\|\widetilde{g}_{n}\|_{\mathrm{Ch},\rho_{n}}^{p}=\int_X (|\widetilde{g}_n(x)|^{p}+(D_{x}(\widetilde{g}_n))^{p}) \rho_{n}(x)\,\ud\mm(x)=1.\]
Since
\[
z_n^p=\int_X (|f_n(x)|^{p}+(D_{x}(f_n))^{p}) \rho_{n}(x)\,\ud\mm(x)=
\int_{E_n} (|f(x)|^{p}+(D_{x}(f))^{p}) \rho(x)\,\ud\mm(x)\
\]
and $\|f\|_{\mathrm{Ch},\rho}=1$, the monotone convergence theorem yields $z_n\uparrow 1$; similarly we obtain $w_n\uparrow 1$.
In our choice of $f$ and $g$ we assumed that
\[\|f-g\|_{\mathrm{Ch},\rho}^{p}=\int_X (|f(x)-g(x)|^{p}+(D_{x}(f-g))^p) \rho(x)\,\ud\mm(x)>\varepsilon^{p}.\]
By Fatou's lemma, using locality of $D_{x}$ to justify pointwise convergence of the integrand, it follows that for $n$ sufficiently large,
\[\|\widetilde{f}_{n}-\widetilde{g}_{n}\|_{\mathrm{Ch},\rho_{n}}=\int_X (|\widetilde{f}_n(x)-\widetilde{g}_n(x)|^{p}+(D_{x}(\widetilde{f}_n - \widetilde{g}_n))^p) \rho_{n}(x)\,\ud\mm(x)>\varepsilon^{p}.\]
Hence, since $\widetilde{f}_{n}, \widetilde{g}_{n} \in {\rm Lip}(X)\cap W^{1,p}_\rho$, the uniform convexity of the norm $\|\cdot\|_{\mathrm{Ch},\rho_{n}}$ on ${\rm Lip}(X)\cap W^{1,p}_\rho$, proved in Step 2, gives
\[ \left\| \frac{\widetilde{f}_{n}+\widetilde{g}_{n}}{2}\right\|_{\mathrm{Ch},\rho_{n}}^{p} =
\int_X \left( \left|\frac{\widetilde{f}_{n}(x)+\widetilde{g}_{n}(x)}{2}\right|^p+\left(D_{x}\left(\frac{\widetilde{f}_{n}+\widetilde{g}_{n}}{2}\right)\right)^p \right) \rho_{n}(x)\ud \mm(x) \leq (1-\delta)^{p}.\]
By using locality of $D$ and the definitions of $f$ and $g$ we obtain
\[\int_{E_n} \left( \left|\frac{f(x)/z_n+g(x)/w_n}{2}\right|^p+\left(D_{x} \left(\frac{f/z_n+g/w_n}{2}\right)\right)^p \right)  \rho(x)\ud \mm(x) \leq (1-\delta)^{p}.\]
By letting $n\to\infty$ and using the dominated convergence theorem we obtain the inequality $\|(f+g)/2\|_{\mathrm{Ch},\rho}\leq (1-\delta)$. Hence the norm $\|\cdot\|_{\mathrm{Ch},\rho}$ is uniformly convex on $W^{1,p}_{\rho}$ and so the Banach space $W^{1,p}_{\rho}$ is reflexive.
\end{proof}

As previously stated (see \eqref{eq:Alica5}) another natural definition of weighted Sobolev space is provided by the closure of Lipschitz functions with respect to the weighted norm.

\begin{defi}[Weighted space $H^{1,p}_\rho(X,\ud,\mm)$]
Let $p>1$ and let $\rho:X\to [0,\infty]$ be a Borel function satisfying $\rho^{-1}\in L^{{1}/{(p-1)}}(\mm)$.
We define $H^{1,p}_{\rho}(X,d,\mm)$ as the closure of Lipschitz functions in $W^{1,p}_{\rho}$, namely
\[ H^{1,p}_{\rho}(X,d,\mm):=\overline{\mathrm{Lip}(X)\cap W^{1,p}_\rho}^{\|\cdot\|_{\rho}}.\]
\end{defi}

As for the $W$ space, we will adopt the shorter notation $H^{1,p}_\rho$.

We can now use the general identification Theorem~\ref{thH=W} and the reflexivity of $W^{1,p}_\rho$
 to prove our first main result, namely that $H^{1,p}_\rho$ coincides with the metric Sobolev space $W^{1,p}(X,\ud,\rho\mm)$ under 
 the $\PI_1$ assumption on $(X,\ud,\mm)$ and no extra assumption on the weight, besides local integrability of $\rho$ and global integrability of $\rho^{1/(1-p)}$.

\begin{teo} \label{thm:main1} Suppose $(X,\ud,\mm)$ is a $\PI_1$ metric measure space.
Let $p>1$ and suppose $\rho:X\to [0,\infty]$ is a Borel function satisfying $\rho\in L^1_{{\rm loc}}(\mm)$ and
$\rho^{-1}\in L^{{1}/{(p-1)}}(\mm)$. Then
$$
W^{1,p}(X,\ud,\rho\mm)=H^{1,p}_\rho.
$$
\end{teo}
\begin{proof} Let $f\in W^{1,p}(X,\ud,\rho\mm)$. By Theorem~\ref{thH=W} and the inequality
$|\nabla g|_w\leq |\nabla g|$ for $g$ Lipschitz,
we can approximate $f$ in $L^p(\rho\mm)$ by
functions $f_n\in \mathrm{Lip}(X)\cap L^p(\rho\mm)$ with $\|f_n\|_\rho$ uniformly bounded
and even $|\nabla f_n|\to |\nabla f|_{p,\rho\mm}$ in $L^p(\rho\mm)$. By reflexivity we have that
$f_n$ weakly converge to $f$ in $W^{1,p}_\rho$, and since $H^{1,p}_\rho$ by definition is a closed subspace it follows
that $f\in H^{1,p}_\rho$ as well. In addition, the weak lower semicontinuity of the norm gives
\begin{equation}\label{eq:first_piece}
\int_X (|f|^p+|\nabla f|_w^p)\rho\ud\mm\leq
\int_X (|f|^p+|\nabla f|_{p,\rho\mm}^p)\rho\ud\mm.
\end{equation}

Conversely, let $f\in H^{1,p}_\rho$ and let $f_n\in \mathrm{Lip}(X)\cap W^{1,p}_\rho$ be convergent to $f$
in $W^{1,p}_\rho$ norm. Using \eqref{eq:may242} of Proposition~\ref{prop:propW11} we obtain that
$$
\lim_{n,\,m\to\infty}\int_X |\nabla (f_n-f_m)|^p\rho\ud\mm=0.
$$
It follows that $\int |\nabla f_n|^p\ud\rho\mm$ is uniformly bounded, therefore one obtains
$f\in H^{1,p}(X,\ud,\rho\mm)$ and therefore, by Theorem \ref{thH=W}, $f\in W^{1,p}(X,\ud,\rho\mm)$.
\end{proof}

\begin{oss}[$W^{1,p}(X,\ud,\rho\mm)$ and $H^{1,p}_\rho$ are isometric]{\rm
The second part of the proof of Theorem~\ref{thm:main1} can be improved if we use the finer information
(see Theorem~12.5.1 in \cite{ShaK}, while \cite{Che} covered only the case $p>1$) that
\begin{equation}\label{eq:second_piece}
|\nabla f|=|\nabla f|_{w}\qquad\text{$\mm$-a.e. in $X$}
\end{equation} 
for all $f\in \mathrm{Lip}(X)\cap W^{1,1}(X,\ud,\mm)$, under the $\PI_1$ assumption (recall $|\nabla f|_w$ stands for $|\nabla f|_{1,\mm}$). 
Indeed, using \eqref{eq:second_piece} one can get 
\begin{equation}\label{eq:third_piece}
\int_X |\nabla f|_{p,\rho\mm}^p\rho\ud\mm\leq \int_X |\nabla f|_w^p\rho\ud\mm
\end{equation}
which, combined with \eqref{eq:first_piece}, gives that the spaces are isometric.
}\end{oss}

\section{Identification of Weighted Sobolev Spaces}\label{sec:4}

In this section we prove that for $p>1$, under certain assumptions on the space $(X,\ud,\mm)$ and on the weight $\rho$, 
the weighted Sobolev spaces $W^{1,p}_\rho$ and $H^{1,p}_\rho$ coincide. Unless otherwise stated, all integrals are with respect to $\mm$. Our 
second main result is the following.

\begin{teo}[Identification of weighted and metric Sobolev spaces]\label{mainthm}
Suppose $(X,\ud,\mm)$ is a $\PI_1$ metric measure space.
Let $p>1$ and $\rho\colon X \to [0,\infty]$ be such that $\rho \in L^1_{\mathrm{loc}}(\mm)$, $\rho^{-1} \in L^{{1}/{(p-1)}}(\mm)$ and
\[L_{\rho}:=\liminf_{n \to \infty} \frac{1}{n^{p}}\left(\int_{X}\rho^{n}\ud\mm\right)^{1/n}\left(\int_{X}\rho^{-n}\ud\mm\right)^{1/n}<\infty.\]
Then $W^{1,p}(X,\ud,\rho\mm)=W^{1,p}_\rho$ and, in particular, $H^{1,p}_{\rho}= W^{1,p}_{\rho}$.
\end{teo}

Suppose that the hypotheses of Theorem \ref{mainthm} hold. Since $L_{\rho}<\infty$, there exists $N\in \mathbb{N}$ such that $\|\rho\|_{L^{N}(\mm)}\|\rho^{-1}\|_{L^{N}(\mm)}<\infty$. The assumptions $\rho \in L^1_{\mathrm{loc}}(\mm)$ and $\rho^{-1} \in L^{{1}/{(p-1)}}(\mm)$ imply that $\rho$ is not identically $0$ or $\infty$. Hence $\|\rho\|_{L^{N}(\mm)}<\infty$ and $\|\rho^{-1}\|_{L^{N}(\mm)}<\infty$. These imply $\mm \{\rho \geq 1\}$ and $\mm \{\rho \leq 1\}$ respectively; hence $\mm(X)<\infty$. Further, the integrability $\rho \in L^{N}(\mm)$ (for some $N$) and $\mm(X)<\infty$ imply $\rho \in L^{1}(\mm)$. We use these facts freely in what follows.

We already know that we can identify $H^{1,p}_\rho$ with $W^{1,p}(X,\ud,\rho\mm)$, thanks to Theorem~\ref{thm:main1}.
Hence, we argue by contradiction, assuming that $H^{1,p}_{\rho}\subsetneq W^{1,p}_{\rho}$, and we derive a contradiction. Since $H^{1,p}_{\rho}$ is a closed subspace of the reflexive space $W^{1,p}_{\rho}$ (Theorem~\ref{reflexivity}) there exists an element $u\in W^{1,p}_{\rho} \setminus H^{1,p}_{\rho}$ such that $\|u+v\|_{\rho} \geq \|u\|_{\rho}$ for any $v \in H^{1,p}_{\rho}$
(it suffices, given $z\in W^{1,p}_\rho\setminus H^{1,p}_\rho$, to minimize $\|z-h\|_\rho$ as $h$ runs in $H^{1,p}_\rho$
and then define $u=z-h$, where $h$ is a minimizer.
Now, suppose $v$ is of the form $v=-tw$ for $t \in (0,1)$, with $w \in H^{1,p}_{\rho}$ and $w=u$ on a Borel set $E \subset X$. Then we obtain,
\begin{multline*} \int_{E} (|u|^p+|\nabla u|_{w}^{p})\rho \ud\mm+ \int_{X\setminus E} (|u|^p+|\nabla u|_{w}^{p})\rho\ud\mm\\
\leq (1-t)^{p}\int_{E} (|u|^p+|\nabla u|_{w}^{p})\rho \ud\mm + \int_{X\setminus E}(|u-tw|^{p}+|\nabla(u-tw)|_{w}^{p})\rho \ud\mm\end{multline*}
and hence
\begin{multline*} \frac{(1-(1-t)^{p})}{t}\int_{E}(|u|^p+|\nabla u|_{w}^{p})\rho \ud\mm\\ 
\leq \int_{X\setminus E} \frac{(|u|+t|w|)^{p}-|u|^{p}+(|\nabla u|_{w}+t|\nabla w|_{w})^{p}-|\nabla u|_{w}^{p}}{t}\rho \ud\mm.\end{multline*}
By letting $t \to 0$ and using the dominated convergence theorem we obtain
\begin{equation}\label{orthogonality}
\int_{E} (|u|^{p}+|\nabla u|_{w}^{p})\rho \ud\mm\leq \int_{X\setminus E}(|u|^{p-1}|w|+|\nabla u|_{w}^{p-1}|\nabla w|_{w})\rho \ud\mm.
\end{equation}

To apply \eqref{orthogonality} we need to use $u\in W^{1,p}_{\rho}$ to construct an appropriate test function in $H^{1,p}_{\rho}$. To do this,
as in the proof of Proposition~\ref{prop:propW11}, we use a maximal operator estimate to obtain Lipschitz bounds on the restriction of $u$ to a smaller subset and then extend this restriction to a Lipschitz map on $X$. 

Now let $u\in W^{1,p}_{\rho} \setminus H^{1,p}_{\rho}$ such that $\|u+v\|_{\rho} \geq \|u\|_{\rho}$ for any $v \in H^{1,p}_{\rho}$ and set 
\[g:=\max \{|u|,M(|\nabla u|_{w})\},\]
where $M$ is the maximal operator w.r.t. $\mm$, defined in \eqref{eq:def_maximal}.
We also define the set
\[F_{\lambda}:=\left\{x\in X:\ \text{$x$ is a Lebesgue point of $u$ and $g(x)\leq \lambda$}\right\}.\]
Arguing as in the proof of Proposition~\ref{prop:propW11} we obtain a Lipschitz function $u_\lambda$ with $|u_\lambda|\leq\lambda$, ${\rm Lip}(u_\lambda)\leq C\lambda$, and equal to $u$ $\mm$-a.e. in $F_\lambda$. Since $\mm(X)<\infty$ and $\rho\in L^{1}(\mm)$, we deduce $u_{\lambda} \in H^{1,p}_\rho$.

Now we apply \eqref{orthogonality} with $w=u_{\lambda}$ and $E=F_{\lambda}$ to obtain
\begin{equation}\label{orthogestimate}
\int_{F_{\lambda}} (|u|^{p}+|\nabla u|_{w}^{p})\rho \ud\mm\leq C\lambda \int_{X\setminus F_{\lambda}}(|u|^{p-1}+|\nabla u|_{w}^{p-1})\rho \ud\mm.
\end{equation}

Next, we prove the following estimate:

\begin{prop}\label{distest}
Let $f\colon (0,\infty) \to (0,\infty)$ be a continuously differentiable and strictly decreasing function such that $f(\lambda)\to 0$
as $\lambda\to \infty$ and $-\int_0^t\lambda f'(\lambda)\ud\lambda<\infty$ for all $t>0$. Then, under assumption
\eqref{orthogestimate}, there exists a constant $C>0$ such that
\begin{align*}
\int_X f(g)(|u|^p+|\nabla u|^p_{w})\rho \ud\mm\leq C \int_{X} \Phi(g)(|u|^{p-1}+|\nabla u|^{p-1}_{w})\rho 
\ud\mm\end{align*}
with $\Phi(t)=-\int_0^t\lambda f'(\lambda)\ud\lambda$.
\end{prop}
\begin{proof} First of all we notice that the statement makes sense, since $g>0$ $\mm$-a.e. on the
set where $|u|^p+|\nabla u|_w^p$ is positive (therefore, understanding the integrand as being null on $\{g=0\}$).
As in  
\cite[Lemma 1]{Zhi}, we can apply Cavalieri's formula $\int \psi\ud\mu=\int_0^\infty \mu(\{\psi>t\})\ud t$
with $\psi$ Borel nonnegative and $\mu$ a finite Borel measure. For $k\in L^1(\mm)$, choosing
$\mu= k\mm$ and $\psi=f(g)$, $\psi=\Phi(g)$ and using the change of variable $\lambda=f^{-1}(t)$ yields
\begin{align}\label{stima1}
&\int_{X} f(g)k= -\int_{0}^{\infty} f'(\lambda)\int_{F_{\lambda}} k\ud\lambda,\\
\label{stima2}
&\int_{X} \Phi(g)k= \int_{0}^{\infty}\Phi'(\lambda)\int_{X\setminus F_{\lambda}} k\ud\lambda
\end{align}
for all $k\in L^1(\mm)$ nonnegative, and eventually for any nonnegative Borel $k$.
Now observe that multiplying inequality \eqref{orthogestimate} by $-f'(\lambda)$ (recall that $-f'(\lambda)\geq 0$ in $(0,\infty)$) and integrating from $0$ to $\infty$ we get 
\begin{align}\label{stima3}
&-\int_0^{\infty}f'(\lambda)\int_{F_{\lambda}}( |u|^p+|\nabla u|_{w}^{p})\rho \ud\lambda\\
& \qquad \leq -C\int_{0}^{\infty}\lambda f'(\lambda)\int_{X\setminus F_{\lambda}}(|u|^{p-1}+|\nabla u|^{p-1}_{w})\rho \ud\lambda \nonumber \\
\nonumber
& \qquad =C\int_{0}^\infty\Phi'(\lambda)\int_{X\setminus F_{\lambda}}(|u|^{p-1}+|\nabla u|_{w}^{p-1})\rho \ud\lambda.
\end{align}
By applying \eqref{stima1} to \eqref{stima3} with $k=(|u|^p+|\nabla u|_{w}^p)\rho$ we get
\begin{align}\label{stima4}
\int_{X}f(g)(|u|^p+|\nabla u|^p_{w})\rho\leq C\int_{0}^\infty\Phi'(\lambda)\int_{X\setminus F_{\lambda}}(|u|^{p-1}+|\nabla u|_{w}^{p-1})\rho\ud\lambda.
\end{align}
Further, we choose $k=(|u|^{p-1}+|\nabla u|_{w}^{p-1}) \rho$ and apply inequality \eqref{stima2} to \eqref{stima4} to obtain the thesis.
\end{proof}

We modify the estimate from Proposition \ref{distest} by using H\"older's inequality for the measure $\rho \mm$; for any non negative Borel function $G$ on $X$ we have,
\[\int_X \Phi(g)G^{p-1} \rho \leq \left(\int_X G^{p}\rho \right)^{1/p'}\left( \int_X\Phi(g)^{p} \rho \right)^{1/p}\]
where $1/p+1/p'=1$. By applying Proposition \ref{distest} then using the above inequality with $G=|u|$ and $G=|\nabla u|_{w}$ we deduce
\begin{align}\label{modifiedestimate}
& \int_X f(g)(|u|^{p}+|\nabla u|_{w}^{p})\rho\\ \nonumber
&\qquad \leq C\left( \left(\int_X |u|^{p}\rho\right)^{1/p'} + \left(\int_X |\nabla u|_{w}^{p}\rho\right)^{1/p'} \right)\left( \int_X \Phi(g)^{p} \rho \right)^{1/p}\\ \nonumber
&\qquad \leq C\left(\int_X (|u|^{p} + |\nabla u|_{w}^{p})\rho\right)^{1/p'}\left( \int_X \Phi(g)^{p} \rho \right)^{1/p}\\ \nonumber
&\qquad = C\|u\|_{\rho}^{p/p'}\left( \int_X \Phi(g)^{p} \rho \right)^{1/p}.
\end{align}

Now we fix $p^*\in (1,p)$ and choose $\varepsilon>0$ such that $\varepsilon< \min\Big\{\frac{1}{2(p-1)}, \frac{1}{p-1}\Big(1-\frac{p^*}{p}\Big)\Big\}$. Therefore,
$(p-1)\varepsilon<1/2$ and $p_\varepsilon:=p(1+(1-p)\varepsilon)>p^*$.
Our next goal is to prove the inequality
\begin{equation}\label{stringA}
\int_X g^{(1-p)\varepsilon}(|u|^{p}+|\nabla u|_{w}^{p})\rho \leq C\varepsilon \|u\|_{\rho}^{p/p'}\left( \int_X g^{p_\varepsilon} \rho\right)^{1/p}.%\\ \nonumber
%&= C\varepsilon \|u\|_{\rho}^{p/p'}\|g\|_{L^{Q}(\rho \mm)}^{Q/p}
\end{equation}
 
In order to prove \eqref{stringA}, let $f(\lambda):=\lambda^{(1-p)\varepsilon}$, so that
\[\Phi(\lambda):=\frac{(p-1)\varepsilon}{1-(p-1)\varepsilon}\lambda^{1-(p-1)\varepsilon}.\]
Now we observe that
$$
\|\Phi \circ g\|_{L^{p}(\rho\mm)}^p= \frac{(p-1)^{p}\varepsilon^{p}}{(1-(p-1)\varepsilon)^{p}}\int_X g^{p(1-(p-1)\varepsilon)}\rho
\leq 2^p(p-1)^p\varepsilon^p\int_X g^{p(1-(p-1)\varepsilon)}\rho,
$$
by our choice of $\varepsilon$.
Hence, by applying \eqref{modifiedestimate} with our choice of $f$ and $\Phi$ we obtain
\eqref{stringA}.

Now, recalling that $g=\max\{|u|,M(|\nabla u|_{w})\}$ and using the triangle inequality, we estimate
\begin{equation}\label{stringB}
\|g\|_{L^{p_\varepsilon}(\rho \mm)}^{p_\varepsilon/p} \leq (\|u\|_{L^{p_\varepsilon}(\rho \mm)}+\|M(|\nabla u|_{w})\|_{L^{p_\varepsilon}(\rho \mm)})^{p_\varepsilon/p}.
\end{equation}
We will use H\"older's inequality and boundedness of the maximal operator to bound the right hand side in terms of $\|u\|_{\rho}$ and $\rho$. 
Notice that the constants $C$ appearing in the estimates below are independent of $\epsilon$, since we are going to apply the maximal estimates
with exponent $p_\varepsilon r$ and $p_\varepsilon r>p^*$, by our choice of $\varepsilon$.
We handle 
$\|u\|_{L^{p_\varepsilon}(\rho \mm)}$ and $\|M(|\nabla u|_{w})\|_{L^{p_\varepsilon}(\rho \mm)}$ separately but with a similar argument; we apply H\"older's inequality twice with the following exponents:
\[r=r_\varepsilon:=\frac{2+(1-p)\varepsilon}{2+2(1-p)\varepsilon},\qquad s=s_\varepsilon:=\frac{2}{2+(1-p)\varepsilon}.\]
It is easy to see that $r, \,s>1$ and that the conjugate H\"older exponents $r',\,s'$ are respectively given by
\[r'=\frac{2+(1-p)\varepsilon}{(p-1)\varepsilon},\qquad
s'=\frac{2}{(p-1)\varepsilon}.\]
Furthermore, these exponents satisfy the equations $p_\varepsilon rs=p$ and $s'=r's$. Now, let us derive the inequalities for $M(|\nabla u|_{w})$; the case of $|u|$ is similar but easier, since we don't need to use boundedness of the maximal operator. 

\begin{align}\label{maxest}
\int_X M(|\nabla u|_{w})^{p_\varepsilon} \rho &\leq \left( \int_X M(|\nabla u|_{w})^{p_\varepsilon r} \right)^{1/r} \left( \int_X \rho^{r'} \right)^{1/r'}\\
\nonumber
&\leq C \left( \int_X |\nabla u|_{w}^{p_\varepsilon r}\rho^{1/s}\rho^{-1/s}\right)^{1/r}\left( \int_X \rho^{r'} \right)^{1/r'}\\
\nonumber
&\leq C \left( \int_X |\nabla u|_{w}^{p_\varepsilon rs}\rho \right)^{1/(rs)} \left( \int_X \rho^{r'}\right)^{1/r'} \left( \int_X \rho^{-s'/s} \right)^{1/(rs')}\\
\nonumber
&\leq C\|u\|_{\rho}^{p/(rs)} \left( \int_X \rho^{r'} \right)^{1/r'} \left( \int_X \rho^{-r'}\right)^{1/(rs')} .
\end{align}
Similarly we obtain
\[ \int |u|^{p_\varepsilon} \rho \leq C\|u\|_{\rho}^{p/(rs)} \left( \int_X \rho^{r'} \right)^{1/r'} \left( \int_X \rho^{-r'}\right)^{1/rs'} .\]
Hence we have
\begin{equation}\label{stringC}
(\|u\|_{L^{p_\varepsilon}(\rho \mm)}+\|M(|\nabla u|_{w})\|_{L^{p_\varepsilon}(\rho \mm)})^{p_\varepsilon/p} \leq C\|u\|_{\rho}^{1/(rs)} 
\left( \int_X \rho^{r'} \right)^{1/(pr')} \left( \int_X \rho^{-r'}\right)^{1/(prs')}.
\end{equation}

By combining our estimates \eqref{stringA}, \eqref{stringB} and \eqref{stringC} we obtain
\begin{align}\label{limitfin}
\int g^{(1-p)\varepsilon}(|u|^{p}+|\nabla u|_{w}^{p})\rho \leq C\varepsilon \|u\|_{\rho}^{\frac{p}{p'}+
\frac{1}{rs}}\left( \int_X \rho^{r'}\right)^{1/pr'}\left( \int_X \rho^{-r'}\right)^{1/prs'}.
\end{align}

As $\varepsilon \to 0$ Fatou's lemma gives
\begin{align}\label{Fat}
\int_X (|u|^p+|\nabla u|_{w}^{p}) \rho \leq \liminf_{\varepsilon \downarrow 0} \int_X g^{(1-p)\varepsilon}(|u|^{p}+|\nabla u|_{w}^{p})\rho.
\end{align}
Therefore in order to estimate $\int_X (|u|^p+|\nabla u|_{w}^{p}) \rho$ from above we can estimate
the right hand side of \eqref{limitfin} as $\varepsilon \downarrow 0$ (notice that $r_\varepsilon,\,s_\varepsilon\downarrow 1$, while
$p_\varepsilon\uparrow p$) in the following lemma.

\begin{lem}
The following inequality holds:
\begin{equation}\label{limit}
\liminf_{\varepsilon \downarrow 0} \varepsilon^p \left( \int_{X} \rho^{r'}\right)^{1/r'}\left( \int_{X} \rho^{-r'}\right)^{1/(rs')}\leq \frac{2^p}{(p-1)^p}L_{\rho}
\end{equation}
where $L_{\rho}$ is defined as in Theorem~\ref{mainthm}.
\end{lem}

\begin{proof}
Setting $\varepsilon=2/[(n+1)(p-1)]$ gives $n=r'$ and $rs'=n(n+1)/(n-1)$. Hence the left hand side of \eqref{limit} is bounded from above by
\begin{align*}
\liminf_{n \to \infty} \frac{2^p}{(p-1)^p(n+1)^p} \left( \int_{X} \rho^{n}\right)^{\frac{1}{n}}\left( \int_{X} \rho^{-n}\right)^{\frac{n-1}{n(n+1)}}.
\end{align*}
Since $\mm(X)<\infty$ and $\rho^{-1}\in L^{1/(p-1)}(\mm)$, it follows that $\|\rho^{-1}\|_{L^n(\mm)}\to \| \rho^{-1}\|_{L^{\infty}(\mm)}$ as $n\to \infty$. Hence if $0<\| \rho^{-1}\|_{L^{\infty}(\mm)}<\infty$ then
\begin{align*}
&\liminf_{n \to \infty} \frac{2^p}{(p-1)^p(n+1)^p} \|\rho\|_{L^{n}(\mm)}\|\rho^{-1}\|_{L^{n}(\mm)}^{\frac{n-1}{n+1}}\\
&\qquad = \liminf_{n \to \infty} \frac{2^p}{(p-1)^p(n+1)^p} \|\rho\|_{L^{n}(\mm)}\|\rho^{-1}\|_{L^{n}(\mm)}\\
&\qquad =\frac{2^p}{(p-1)^p}L_{\rho}
\end{align*}
and the thesis follows. If $\|\rho^{-1}\|_{L^{\infty}(\mm)}=\infty$ then $\|\rho^{-1}\|_{L^{n}(\mm)}>1$ for sufficiently large $n$ and we can use the trivial inequality
\[\|\rho^{-1}\|_{L^{n}(\mm)}^{\frac{n-1}{n+1}}< \|\rho^{-1}\|_{L^{n}(\mm)}\]
for every $n$ to obtain the same bound. The case $\|\rho^{-1}\|_{L^{\infty}(\mm)}=0$ is impossible since $\rho \in L^{1}(\mm)$.
\end{proof}

Using \eqref{limitfin}, \eqref{Fat}, \eqref{limit} and the fact that $p/p'=p-1$ we get
\[ \|u\|_{\rho}^{p}\leq CL_{\rho}^{1/p}\|u\|_{\rho}^{p}\]
for some structural constant $C$. Since $u\neq 0$, we obtain
\begin{equation}\label{ineqcontra}
1\leq CL_{\rho}.
\end{equation}
As in \cite{Zhi}, the strategy is now to use the fact that $C$ is independent of $\rho$ to derive a contradiction. 
Following the notation of Theorem \ref{mainthm} we write
\[L_{\eta}=\liminf_{n\to\infty} \frac{1}{n^p} \|\eta \|_{L^{n}(\mm)}\|\eta^{-1}\|_{L^{n}(\mm)}\]
for any Borel $\eta\colon X \to [0,\infty]$. 

%We now show that for any $\varepsilon>0$ there exists a weight $\widetilde{\rho}$ satisfying the hypotheses of Theorem \ref{mainthm} such that $L_{\widetilde{\rho}}\leq \varepsilon$, $W^{1,p}_{\widetilde{\rho}}=W^{1,p}_{\rho}$ and $H^{1,p}_{\widetilde{\rho}}=H^{1,p}_{\rho}$; afterwards we will observe this gives a contradiction so the assumption that $W^{1,p}_{\rho} \setminus H^{1,p}_{\rho} \neq \varnothing$ was false and we obtain Theorem \ref{mainthm}.

\begin{prop}\label{newweight}
For all $\delta>0$ there exists a weight $\widetilde{\rho}\colon X \to [0,\infty)$ satisfying the hypotheses of Theorem~\ref{mainthm} 
such that $L_{\widetilde{\rho}}\leq\delta$, $W^{1,p}_{\rho}=W^{1,p}_{\widetilde{\rho}}$ and $H^{1,p}_{\rho}=H^{1,p}_{\widetilde{\rho}}$. 
\end{prop}

\begin{proof}
Throughout this proof we simply write bounded or unbounded instead of essentially bounded (i.e. in $L^{\infty}(\mm)$) or essentially unbounded (not in $L^{\infty}(\mm)$). We start by observing that if both $\rho$ and $\rho^{-1}$ are bounded then $L_{\rho}=0$, so there is nothing to prove.

If $\rho$ and $\rho^{-1}$ are both unbounded then for every $t>0$ we define
\begin{equation}\label{modweight}
\rho_t(x):=\begin{cases}
\displaystyle{\frac{1}{t}\rho(x)} & \mbox{if}\ \rho(x)\leq 1; \\\\
t\rho(x) & \mbox{if}\ \rho(x) > 1 .
\end{cases}
\end{equation}
Clearly $\rho_t$ satisfies the assumptions of Theorem~\ref{mainthm} since it is bounded by constant multiples of $\rho$.
For the same reason, the $W$ and $H$ weighted Sobolev spaces induced by $\rho$ and $\rho_t$ are the same.
Observe that
\begin{equation}\label{eq:may16}
\int_{X}\rho_t^n \leq t^{-n}\mm(X)+t^n\int_{X}\rho^n,\qquad
\int_{X}\rho_t^{-n} \leq t^{-n}\mm(X)+t^n\int_{X}\rho^{-n}.
\end{equation}
Since $\rho$ is unbounded, we get
\begin{align}\label{keyest}
\lim_{n\to \infty}\frac{1}{\mm(X)}\lambda^{-n}\int_X \rho^n =\infty\qquad\forall\lambda>1.
\end{align}
Using \eqref{keyest} with $\lambda=t^{-2}$ we get
\[
t^{-n}\mm(X)\leq t^n\int_X \rho^n\qquad\text{for $n$ sufficiently large.}
\]
Therefore the first inequality in \eqref{eq:may16} gives
\begin{align}\label{estrho}
\| \rho_t\|_{L^n(\mm)}\leq 2^{\frac{1}{n}}t\| \rho\|_{L^n(\mm)}\qquad\text{for $n$ sufficiently large.}
\end{align}
Arguing in the same way using the unboundedness of $\rho^{-1}$ and the second inequality in
\eqref{eq:may16} we obtain 
\begin{align}\label{estrhomin}
\| \rho_t^{-1}\|_{L^n(\mm)}\leq 2^{\frac{1}{n}}t\| \rho^{-1}\|_{L^n(\mm)}\qquad\text{for $n$ sufficiently large.}
\end{align} 
Hence, by \eqref{estrho} and \eqref{estrhomin} we obtain the thesis choosing $t>0$ with $2t^2<C^{-1}$.

Let us now assume that $\rho$ is bounded but $\rho^{-1}$ is unbounded. For every $t>0$ we define
\begin{align}\label{modweight2}
\rho_t(x):=\bigg \{
\begin{array}{rl}
\frac{1}{t}\rho(x) & \mbox{if}\ 0\leq \rho(x) \leq t \\
 \rho(x) & \mbox{if}\ \rho(x)>t. \\
\end{array}
\end{align}
As before, $\rho_{t}$ satisfies the hypotheses of Theorem \ref{mainthm}. We observe,
\begin{align*}
\int_{X}\rho_t^{-n}&=t^n\int_{\{\rho\leq t\}}\rho^{-n}+\int_{\{\rho>t\}}\rho^{-n}\\
\nonumber
&\leq t^n\int_X \rho^{-n}+ t^{-n}\mm(X).
\end{align*}
Since $\rho^{-1}$ is unbounded proceeding as in \eqref{estrho} we obtain
\begin{align}\label{wr}
\| \rho_{t}^{-1}\|_{L^n(\mm)}\leq 2^{\frac{1}{n}}t \| \rho^{-1}\|_{L^n(\mm)},
\end{align}
while
\begin{align}\label{limsup}
\|\rho_t\|_{L^n(\mm)}\leq \|\rho\|_{L^{\infty}(\mm)}+1
\end{align}
for every $n$. Because $\rho$ is bounded we have
\begin{align}\label{limite}
L_{\rho}=\liminf_{n\to \infty}\frac{\| \rho\|_{L^{\infty}(\mm)}\| \rho^{-1}\|_{L^n(\mm)}}{n^p}.
\end{align} 
Putting together \eqref{wr}, \eqref{limsup} and \eqref{limite} we get
\[
\liminf_{n\to\infty} \frac{1}{n^p} \|\rho_{t} \|_{L^{n}(\mm)}\|\rho_{t}^{-1}\|_{L^{n}(\mm)} \leq 2t L_{\rho}(1+\|\rho\|_{L^{\infty}(\mm)}^{-1})
\]
and we conclude, again, choosing $t>0$ sufficiently small.
The case where $\rho^{-1}$ is bounded and $\rho$ is unbounded is analogous.
\end{proof}

We can now conclude the proof of Theorem \ref{mainthm}. Choose $\delta\in (0,C^{-1})$, where $C>0$ is the constant in \eqref{ineqcontra} and apply 
Proposition~\ref{newweight} to find a weight $\widetilde{\rho}$ satisfying the hypotheses of Theorem~\ref{mainthm} such that $L_{\widetilde{\rho}}\leq\delta$, $W^{1,p}_{\rho}=W^{1,p}_{\widetilde{\rho}}$ and $H^{1,p}_{\rho}=H^{1,p}_{\widetilde{\rho}}$. Then the assumption that $W^{1,p}_{\rho}\setminus H^{1,p}_{\rho} \neq \varnothing$ implies $W^{1,p}_{\widetilde{\rho}} \setminus H^{1,p}_{\widetilde{\rho}}\neq \varnothing$ and we may repeat all of our arguments to obtain an analogue of \eqref{ineqcontra} with $\widetilde{\rho}$ in place of $\rho$; hence,
\[1\leq CL_{\widetilde{\rho}}\leq C\delta<1\]
which gives a contradiction.

\section{Examples and Extensions}\label{sec:5}

In this section we discuss some examples and generalize our results by considering Muckenhoupt weights or by requiring a weaker Poincar\'{e} inequality.

\subsection{An example where $W^{1,p}_\rho\supsetneq H^{1,p}_\rho$}

Let us consider the standard Euclidean structure $X=$ closed unit ball of $\R^2$, $\ud=$ Euclidean distance, $\mm=\Leb{2}$. In \cite{ChiSC} examples
of weights $\rho\in L^1(\mm)$ with $\rho^{-1}\in L^{1/(p-1)}(\mm)$ and $H^{1,p}_\rho\subsetneq W^{1,p}_\rho$ are given for any
$p>1$. Here we report only the example with $p=2$, with a weight $\rho$ in all $L^q$ spaces having also the inverse in all
$L^q$ spaces, $1\leq q<\infty$.

Let $\Omega=B(0,1)\subset\bbR^2$ and $\epsilon\in (0,\pi/2)$. Set \[S_{\epsilon}:=\{(x_1,x_2)\in\Omega\ |\ \tan(\epsilon)< \frac{x_2}{x_1}<\tan(\frac{\pi}{2}-\epsilon)\},\]
$S_{\epsilon}^+:=S_{\epsilon}\cap\{x_2>0\}$ and $S_{\epsilon}^{-}:=S_{\epsilon}\cap \{x_2<0\}$. Let us consider $\rho:\bbR^2\setminus \{0\}\to [0,\infty)$ defined by 
\[
\rho(x):=
\begin{cases}
\left(\ln^{-2}\left(\frac{e}{|x|}\right)\right)^{k\arccos\left(\frac{x_1}{|x_1|}\right)}& \mbox{if}\  0<|x|\leq 1\\
1& \mbox{if}\  |x|>1,
\end{cases}
\]
where $k:\bbR\to [-1,0]$ is a $\pi-$periodic smooth function such that $k'(0)=0$ and 
\[
k\equiv -1\quad \mbox{in}\ \left(\epsilon, \frac{\pi}{2}-\epsilon\right), \qquad k\equiv 0\quad  \mbox{in}\ \left(\frac{\pi}{2},\pi\right).
\]
It follows that $\lambda\in C^0(\bbR^2\setminus \{0\})$ and that
\[
\rho,\rho^{-1}\in \bigcap_{q\in [1,\infty)} L^{q}_{\mathrm{loc}}(\bbR^2).
\]
It is proved by a direct calculation in \cite{ChiSC} that the function
$$
u(x):=
\begin{cases} 1 &\text{if $x_1>0$, $x_2>0$}\cr
 0 &\text{if $x_1<0$, $x_2<0$}\cr
  x_2/|x| &\text{if $x_1<0$, $x_2>0$}\cr
  x_1/|x| &\text{if $x_1>0$, $x_2<0$}
\end{cases}
$$
belongs to $W^{1,2}_{\rho}\setminus H^{1,2}_{\rho}$.

\subsection{Muckenhoupt weights and the $\PI_p$ condition}

An important class of weights is the one introduced by Muckenhoupt \cite{Muk} to study the boundedness of the maximal operator in $L^p$ spaces.
In Euclidean spaces, for this $p$-dependent class of weights $\rho$ it is known that $W^{1,p}_{\rho}=H^{1,p}_{\rho}$,
 see for instance \cite{ChSc1,HKM93} and Theorem~\ref{thm:main3} below. Let us recall the definition of Muckenhoupt weight in the context of metric measure spaces.
 
\begin{defi}[Muckenhoupt weight]
Let $(X,\ud,\mm)$ be a metric measure space and let $\rho:X\to [0,\infty]$ be locally integrable. 
For $p>1$, we say that $\rho$ is an $A_p$-weight if 
\[
\sup_B\left( \ave_{B} \rho \ud\mm \right) \left( \ave_{B} \rho^{-1/(p-1)}\ud\mm\right)^{p-1} <\infty
\]
where the supremum runs among all balls $B$. We say that $\rho$ is an $A_{1}$-weight if there exists a constant $C$ such that
\[
\ave_B \rho\ud\mm\leq C\,{\rm ess}\inf_B u
\]
for all balls $B\subset X$. We denote the class of $A_{p}$ weights on $(X,\ud,\mm)$ by $A_{p}(\mm)$.
\end{defi} 

It is immediate to see, using the H\"older inequality, that $\bigl(\ave_{B} \rho \ud\mm\bigr)\bigl(\ave_{B} \rho^{-1/(p-1)}\ud\mm\bigr)^{p-1}$
is always larger than $1$.
This easily yields that $(X,\ud,\rho\mm)$ is doubling whenever $(X,\ud,\mm)$ is doubling; indeed, for $p>1$,
$$
\ave_{2B}\rho\ud\mm\leq C\biggl(\ave_{2B}\rho^{-1/(p-1)}\ud\mm\biggr)^{1-p}\leq
C\biggl(\ave_{B}\rho^{-1/(p-1)}\ud\mm\biggr)^{1-p}\leq C\ave_{B} \rho\ud\mm
$$ 
and a similar argument works for $p=1$. It follows that the maximal operator with respect to $\rho\mm$ is bounded in $L^p(\rho\mm)$ for all $p>1$. A remarkable fact, proved in the Euclidean case
by Muckenhoupt \cite{Muk}, with a proof that extends readily also to doubling metric measure spaces $(X,\ud,\mm)$ (see \cite[Theorem 9]{ST}), is the fact that even the
maximal operator $M$ in \eqref{eq:def_maximal}, namely the maximal operator respect to $\mm$, is bounded in $L^p(\rho\mm)$ for all $p>1$, and weakly bounded if $p=1$. 

%The boundedness of the maximal operator together with the estimate \eqref{eq:localip} provide a good approximation
%with Lipschitz functions as in Proposition~\ref{prop:propW11}, so that:\footnote{We could check this (maybe without putting a proof), 
%and also check whether the proof in \cite{ChSc1} is substantially different}

%\begin{teo}\label{thm:main3}
%If $(X,\ud,\mm)$ is a $\PI_1$ space and $\rho\in A_p(\mm)$ with $\rho^{-1}\in L^{1/(p-1)}(\mm)$, then $W^{1,p}_\rho=H^{1,p}_\rho$.
%\end{teo}

For $p>1$ it is well known that an $A_{p}$ weight $\rho$ on a Euclidean space is $p$-admissible \cite{HKM93}; this means that the weighted space $(\mathbb{R}^{n},|\cdot |,\rho\mathcal{L}^{n})$ satisfies $\PI_{p}$. A generalization in the metric setting is proved in \cite{Bjorn}. A converse holds in dimension one but it is an open problem for higher dimensions \cite{BBK06}. 

%\begin{prop} [Invariance of $\PI_p$] If $(X,\ud,\mm)$ is a $\PI_p$ space for some $p> 1$ and $\eta\in A_p(\mm)$, then $(X,\ud,\eta\mm)$ is also a $\PI_p$ space.
%\end{prop}
%\begin{proof} Let $\widetilde\mm=\eta\mm$. For $u$ locally Lipschitz, we start from
%\begin{equation}\label{eq:localip2}
%|u(x)-u(y)|\leq C \ud(x,y) (M^{1/p}_{2\Lambda r}(|\nabla u|^p)(x)+M^{1/p}_{2\Lambda r}(|\nabla u|^p)(y))\qquad \ud(x,y)<r
%\end{equation}
%obtained from \eqref{eq:localiplip} by replacing weak gradients with slope. Recall that $M_s$ denotes the maximal operator
%w.r.t. $\mm$ on scale $s$; using multiplication by a cut-off function it is easy to check that boundedness of
%of $M$ in $L^p(\tilde\mm)$ yields the localized $L^p$ boundedness
%$$
%\int_{B(z,s)}(M_s g)^p\ud\tilde\mm\leq C\int_{B(z,2s)} g^p\ud\tilde\mm.
%$$
%We multiply both sides of \eqref{eq:localip2} by $\eta(x)\eta(y)$, integrate w.r.t. $\mm\times\mm$, divide by $(\widetilde\mm(B(x,r)))^2$ and 
%eventually use the boundedness of $M$ in $L^p(\widetilde\mm)$ for $p>1$ to get
%$$
%\ave_{B(x,r)}|u(x)-\ave_{B(x,r)} u\ud\widetilde\mm|\ud\widetilde\mm(x)\leq
%Cr\biggl(\ave_{B(x,2r)}|\nabla u|^p\ud\widetilde\mm\biggr)^{1/p},
%$$
%where now the integrals are averaged with respect to $\widetilde\mm$. This proves the result. 
%\end{proof} 

Let us now compare the Muckenhoupt condition with the Zhikov one, introduced in \cite{Zhi} and used also in the present paper.
\begin{defi}
Let $\rho\colon X \to [0,\infty]$ be Borel and let $p> 1$. We say that $\rho$ belongs to the class $Z_{p}(\mm)$ if
\[\liminf_{n \to \infty} \frac{1}{n^{p}}\|\rho\|_{L^{n}(\mm)}\|\rho^{-1}\|_{L^{n}(\mm)}<\infty.\]
\end{defi}
Even though both the Muckenhoupt and Zhikov conditions lead to the identification of the weighted Sobolev spaces, the following simple examples show that
they are not comparable, even in the Euclidean case. One of the reasons is that the class $A_p$ involves a more local condition; for instance there is no
reason for $(X,\ud,\rho\mm)$ to be doubling when $(X,\ud,\mm)$ is doubling and $\rho\in Z_p(\mm)$. As a matter of fact, 
the Zhikov condition is easier to check. For instance, if both $\exp(t\rho)$ and  $\exp(t \rho^{-1})$ belong
to $L^{1}(\mm)$ for some $t>0$, then $\rho \in Z_{2}(\mm)$ (see the simple proof in \cite{Zhi}, still valid in the metric measure
setting).

\begin{es}{\rm 
Let $X$ be the unit ball of $\R^2$ and let $\rho\colon X \to \mathbb{R}$ be given by
\begin{align*}
\rho(x):=\bigg \{
\begin{array}{rl}
\log(1/|x|) & \mbox{if}\ x_{1}x_{2}>0,\\
{\displaystyle{\frac{1}{\log(1/|x|)}}} & \mbox{if}\ x_{1}x_{2}<0.\\
\end{array}
\end{align*}
It is easy to check that $\exp(\rho),\ \exp(\rho^{-1}) \in L^1(\Leb{2})$ (see also \cite{Zhi}), which implies $\rho \in Z_{2}(\Leb{2})$. 
Further, one can easily prove that $\rho\notin A_{p}(\Leb{2})$ for any $p\geq 1$; indeed, the average of both $\rho$ and $\rho^{-1}$ 
on balls centred at the origin tends to infinity as the radius of the ball tends to zero.
}\end{es}

\begin{es} {\rm Let $X=(0,1)$. Then, by a direct computation, $|x|^{\alpha} \in A_{p}(\Leb{1})$ provided $-1<\alpha<p-1$. 
Hence $1/\sqrt{x} \in A_{p}(\Leb{1})$  for any $p\geq 1$; however, $1/\sqrt{x} \notin Z_{p}(\Leb{1})$ for any $p\geq 1$ since $1/\sqrt{x}\notin L^{n}(\Leb{1})$ for any $n\geq 2$.}
\end{es}
 
 \subsection{Relaxation of the $\PI_1$ assumption to $\PI_p$.} Suppose $\rho^{-1}\in L^\alpha(\mm)$ for some exponent $\alpha\in (1/(p-1),\infty]$ and set $q=p\alpha/(\alpha+1)$ ($q=p$ if $\alpha=\infty$). Then the definition
\begin{equation}\label{eq:Alica2}
W^{1,p}_{\rho,q}:=\left\{f\in W^{1,q}(X,\ud,\mm):\ \int_X |f|^p\rho\ud\mm+\int_X|\nabla f|_{q,\mm}^p\rho\ud\mm<\infty\right\},
\end{equation}
with the corresponding norm, and the corresponding definition of $H^{1,p}_{\rho,p}$ (namely the closure
of $\mathrm{Lip}(X)\cap W^{1,p}_{\rho,q}$ inside $W^{1,p}_{\rho,q}$) are much more natural.
Indeed, there is already a natural embedding in $W^{1,q}(X,\ud,\mm)$ that is missing in the general case, so there is no necessity to invoke the space
$W^{1,1}(X,\ud,\mm)$ and the $\PI_1$ structure (notice that $q=1$ corresponds precisely to $\alpha=1/(p-1)$). 
The embedding provides completeness of $W^{1,p}_{\rho,q}$, via the completeness of $W^{1,q}(X,\ud,\mm)$, 
and also the proof of reflexivity can be immediately adapted to the space $W^{1,p}_{\rho,q}$.

Assume now that $\rho^{-1}\in L^\alpha(\mm)$ for any $\alpha\in (1/(p-1),\infty)$, as it happens when $L_\rho<\infty$; in this case we can
choose the power $q$ in \eqref{eq:Alica2} as close to $p$ as we wish, and use the fact that $\PI_p$ is an open ended condition
to choose $q$ in such a way that $\PI_q$ still holds. This leads to the following result 
(which also shows that the space $W^{1,p}_{\rho,q}$ is essentially independent of the exponent $q$).

\begin{teo} \label{thm:main4} Let $(X,\ud,\mm)$ be a $\PI_p$ metric measure space. Let $\rho\in L^1_{\mathrm{loc}}(\mm)$ be a nonnegative Borel function satisfying $\rho^{-1}\in L^\alpha(\mm)$ for all $\alpha\in (1/(p-1),\infty)$. Then the space $W^{1,p}_{\rho,q}$ in \eqref{eq:Alica2}
and its norm do not depend on the choice of $q\in [1,p)$, as soon as $\PI_q$ holds, and
\begin{equation}\label{eq:Alica5}
H^{1,p}_{\rho,q}=W^{1,p}(X,\ud,\rho\mm)\qquad\text{whenever $\PI_q$ holds.}
\end{equation}
If, in addition, $L_\rho<\infty$, then $H^{1,p}_{\rho,q}=W^{1,p}_{\rho,q}$. 
\end{teo}
\begin{proof} Recall that (see Theorem~12.5.1 in \cite{ShaK}, while \cite{Che} covered only the case $q>1$) 
$$
|\nabla f|=|\nabla f|_{q,\mm}\qquad\text{$\mm$-a.e. in $X$, for all $f\in\mathrm{Lip}(X)\cap W^{1,q}(X,\ud,\mm)$}
$$
for all $q\in [1,\infty)$, under the $\PI_q$ condition. 
This, combined with the locality of weak gradients and the Lusin approximation with Lipschitz functions gives
$$
|\nabla f|_{q,\mm}=|\nabla f|_{q',\mm}\qquad\text{$\mm$-a.e. in $X$}
$$
whenever $f\in W^{1,q}(X,\ud,\mm)\cap W^{1,q'}(X,\ud,\mm)$ and both $\PI_q$ and $\PI_{q'}$ hold.
Then, the independence of $W^{1,p}_{\rho,q}$ with respect to $q$ follows by Proposition~\ref{prop:propW11bis}.
The identity \eqref{eq:Alica5} and the last statement can be obtained repeating respectively the proofs of Theorem~\ref{thm:main1} 
and Theorem~\ref{mainthm} with this new class of weighted spaces.
\end{proof}

 \subsection{Combination of Zhikov and Muckenhoupt weights}
 
Zhikov \cite{Zhi} proves identification of weighted Sobolev spaces for weights $\rho$ expressable as a product $\rho=\rho_{M}\rho_{Z}$ where $\rho_{M} \in A_{p}(\mathcal{L}^{n})$ and $\rho_{Z} \in Z_{p}(\rho_{M}\mathcal{L}^{n})$. The minor adaptations needed to include Muckenhoupt weights work also
in the metric setting.

\begin{teo}\label{ZhiMo}
Suppose $(X,\ud,\mm)$ is a $\PI_{1}$ metric measure space. Let $p>1$ and $\rho=\rho_{M}\rho_{Z}$ where $\rho_{M} \in A_{p}(\mm)$ and $\rho_{Z}\in Z_{p}(\rho_{M}\mm)$. If $\rho \in L^{1}_{\mathrm{loc}}(\mm)$ and $\rho^{-1}\in L^{1/(p-1)}(\mm)$ then $W^{1,p}_{\rho}(\mm)=H^{1,p}_{\rho}(\mm)$.
\end{teo}

\begin{proof}
%The Muckenhoupt and Zhikov conditions imply $\rho_{M}\in L^{1}_{loc}(\mm)$ and $\rho_{Z}^{-n}\rho_{M}\in L^{1}_{loc}(\mm)$ for sufficiently large $n$; by using the H\"older's inequality these give $\rho\in L^{1}_{loc}(\mm)$. Using open endedness of the Muckenhoupt condition (\cite{elisa}) we know that $\rho_{M} \in A_{p-\varepsilon}(\mm)$ for some $\varepsilon>0$ so that $\rho_{M}^{-1} \in L^{1/(p-1-\varepsilon)}_{loc}(\mm)$. By combining this with $\rho_{Z}^{-n}\rho_{M}\in L^{1}_{loc}(\mm)$ for large $n$ we obtain $\rho^{-1}\in L^{1/(p-1)}(\mm)$.

As remarked in the discussion before Theorem \ref{thm:main3} we know that, the maximal operator with respect to $\mm$ is bounded in $L^{p}(\rho_{M}\mm)$ if $p>1$. To obtain the identification $W^{1,p}_{\rho}(\mm)=H^{1,p}_{\rho}(\mm)$ we apply exactly the same argument as in the proof of Theorem~\ref{mainthm} apart from the fact that in \eqref{maxest} we apply H\"older's inequality and boundedness of the maximal operator with respect to the measure $\rho_{M}\mm$; hence \eqref{maxest} changes to
\begin{align*}
\int_X M(|\nabla u|_{w})^{p_\varepsilon} \rho_Z \rho_M &\leq \left( \int_X M(|\nabla u|_{w})^{p_\varepsilon r} \rho_M\right)^{1/r} \left( \int_X \rho_Z^{r'}\rho_M \right)^{1/r'}\\
\nonumber
&\leq C \left( \int_X |\nabla u|_{w}^{p_\varepsilon r}\rho_Z^{1/s}\rho_{Z}^{-1/s}\rho_M\right)^{1/r}\left( \int_X \rho_Z^{r'}\rho_M \right)^{1/r'}\\
\nonumber
&\leq C \left( \int_X |\nabla u|_{w}^{p_\varepsilon rs}\rho_Z \rho_M\right)^{1/(rs)} \left( \int_X \rho_Z^{r'}\rho_M\right)^{1/r'} \left( \int_X \rho_Z^{-s'/s} \rho_M\right)^{1/(rs')}\\
\nonumber
&\leq C\|u\|_{\rho}^{p/(rs)} \left( \int_X \rho_Z^{r'} \rho_M\right)^{1/r'} \left( \int_X \rho_Z^{-r'}\rho_M\right)^{1/(rs')} .
\end{align*}
With this estimate we are able to use the assumption $\rho_{Z}\in Z_{p}(\rho_{M}\mm)$ to obtain again the identification $W^{1,p}_{\rho}(\mm)=H^{1,p}_{\rho}(\mm)$.
\end{proof}

The following result, which is well known in the Euclidean setting, easily follows from Theorem \ref{ZhiMo} by taking $\rho_Z=1$.
\begin{cor}\label{thm:main3}
Suppose $(X,\ud,\mm)$ is a $\PI_1$ space. Let $\rho\in A_p(\mm)$ with $\rho\in L^{1}_{\mathrm{loc}}(\mm)$ and $\rho^{-1}\in L^{1/(p-1)}(\mm)$. Then we have $W^{1,p}_\rho(\mm)=H^{1,p}_\rho(\mm)$.
\end{cor}

A more direct proof of
Corollary \ref{thm:main3} can be obtained using the same approach described in Proposition \ref{prop:propW11}. More precisely, 
\begin{align*}
\ave_{B} |f|\ud\mm&\leq \frac{1}{\mm(B)}\left(\int_{B}|f|^p\rho\ud\mm\right)^{\frac{1}{p}}\left(\int_{B}\rho^{-\frac{1}{p-1}}\ud\mm\right)^{\frac{p}{p-1}}\\
\nonumber
&\leq \left(\ave_{B} |f|^p\rho\ud\mm\right)^{\frac{1}{p}}\left(\ave_B\rho\ud\mm\left(\ave_B\rho^{-\frac{1}{p-1}}\ud\mm\right)^{p-1}\right)^{\frac{1}{p}}
\end{align*} 
and since $\rho\in A_p(\mm)$ we get
\begin{align}\label{stimaMax}
M_{\mm} f(x)\leq c \left(M_{\rho\mm} f^p(x)\right)^{\frac{1}{p}}
\end{align}
where $c>0$ and depends only on $\rho$. Given $f\in W^{1,p}_{\rho}\subset W^{1,1}(X,\ud,\mm)$, by \eqref{eq:localip} and \eqref{stimaMax} we have
\begin{align*}
|f(x)-f(y)|\leq C\ud(x,y)\left( \left(M_{\rho\mm} |\nabla f|_{1,\mm}^p(x)\right)^{\frac{1}{p}}+\left(M_{\rho\mm} |\nabla f|_{1,\mm}^p(y)\right)^{\frac{1}{p}}\right)
\end{align*}
and, proceeding exactly as in Proposition \ref{prop:propW11}, we obtain $W^{1,p}_{\rho}\subset H^{1,p}_{\rho}$ and therefore $W^{1,p}_{\rho}=H^{1,p}_{\rho}$.

\end{document}